\documentclass[12pt]{amsart}

\usepackage[draft]{say}

\usepackage{etoolbox}

\makeatletter
\let\old@tocline\@tocline
\let\section@tocline\@tocline
\newcommand{\subsection@dotsep}{4.5}
\newcommand{\subsubsection@dotsep}{4.5}
\patchcmd{\@tocline}
  {\hfil}
  {\nobreak
     \leaders\hbox{$\m@th
        \mkern \subsection@dotsep mu\hbox{.}\mkern \subsection@dotsep mu$}\hfill
     \nobreak}{}{}
\let\subsection@tocline\@tocline
\let\@tocline\old@tocline

\patchcmd{\@tocline}
  {\hfil}
  {\nobreak
     \leaders\hbox{$\m@th
        \mkern \subsubsection@dotsep mu\hbox{.}\mkern \subsubsection@dotsep mu$}\hfill
     \nobreak}{}{}
\let\subsubsection@tocline\@tocline
\let\@tocline\old@tocline

\let\old@l@subsection\l@subsection
\let\old@l@subsubsection\l@subsubsection

\def\@tocwriteb#1#2#3{%
  \begingroup
    \@xp\def\csname #2@tocline\endcsname##1##2##3##4##5##6{%
      \ifnum##1>\c@tocdepth
      \else \sbox\z@{##5\let\indentlabel\@tochangmeasure##6}\fi}%
    \csname l@#2\endcsname{#1{\csname#2name\endcsname}{\@secnumber}{}}%
  \endgroup
  \addcontentsline{toc}{#2}%
    {\protect#1{\csname#2name\endcsname}{\@secnumber}{#3}}}%

\newlength{\@tocsectionindent}
\newlength{\@tocsubsectionindent}
\newlength{\@tocsubsubsectionindent}
\newlength{\@tocsectionnumwidth}
\newlength{\@tocsubsectionnumwidth}
\newlength{\@tocsubsubsectionnumwidth}
\newcommand{\settocsectionnumwidth}[1]{\setlength{\@tocsectionnumwidth}{#1}}
\newcommand{\settocsubsectionnumwidth}[1]{\setlength{\@tocsubsectionnumwidth}{#1}}
\newcommand{\settocsubsubsectionnumwidth}[1]{\setlength{\@tocsubsubsectionnumwidth}{#1}}
\newcommand{\settocsectionindent}[1]{\setlength{\@tocsectionindent}{#1}}
\newcommand{\settocsubsectionindent}[1]{\setlength{\@tocsubsectionindent}{#1}}
\newcommand{\settocsubsubsectionindent}[1]{\setlength{\@tocsubsubsectionindent}{#1}}

\renewcommand{\l@section}{\section@tocline{1}{\@tocsectionvskip}{\@tocsectionindent}{}{\@tocsectionformat}}%
\renewcommand{\l@subsection}{\subsection@tocline{2}{\@tocsubsectionvskip}{\@tocsubsectionindent}{}{\@tocsubsectionformat}}%
\renewcommand{\l@subsubsection}{\subsubsection@tocline{3}{\@tocsubsubsectionvskip}{\@tocsubsubsectionindent}{}{\@tocsubsubsectionformat}}%
\newcommand{\@tocsectionformat}{}
\newcommand{\@tocsubsectionformat}{}
\newcommand{\@tocsubsubsectionformat}{}
\expandafter\def\csname toc@1format\endcsname{\@tocsectionformat}
\expandafter\def\csname toc@2format\endcsname{\@tocsubsectionformat}
\expandafter\def\csname toc@3format\endcsname{\@tocsubsubsectionformat}
\newcommand{\settocsectionformat}[1]{\renewcommand{\@tocsectionformat}{#1}}
\newcommand{\settocsubsectionformat}[1]{\renewcommand{\@tocsubsectionformat}{#1}}
\newcommand{\settocsubsubsectionformat}[1]{\renewcommand{\@tocsubsubsectionformat}{#1}}
\newlength{\@tocsectionvskip}
\newcommand{\settocsectionvskip}[1]{\setlength{\@tocsectionvskip}{#1}}
\newlength{\@tocsubsectionvskip}
\newcommand{\settocsubsectionvskip}[1]{\setlength{\@tocsubsectionvskip}{#1}}
\newlength{\@tocsubsubsectionvskip}
\newcommand{\settocsubsubsectionvskip}[1]{\setlength{\@tocsubsubsectionvskip}{#1}}

\patchcmd{\tocsection}{\indentlabel}{\makebox[\@tocsectionnumwidth][l]}{}{}
\patchcmd{\tocsubsection}{\indentlabel}{\makebox[\@tocsubsectionnumwidth][l]}{}{}
\patchcmd{\tocsubsubsection}{\indentlabel}{\makebox[\@tocsubsubsectionnumwidth][l]}{}{}

\newcommand{\@sectypepnumformat}{}
\renewcommand{\contentsline}[1]{%
  \expandafter\let\expandafter\@sectypepnumformat\csname @toc#1pnumformat\endcsname%
  \csname l@#1\endcsname}
\newcommand{\@tocsectionpnumformat}{}
\newcommand{\@tocsubsectionpnumformat}{}
\newcommand{\@tocsubsubsectionpnumformat}{}
\newcommand{\setsectionpnumformat}[1]{\renewcommand{\@tocsectionpnumformat}{#1}}
\newcommand{\setsubsectionpnumformat}[1]{\renewcommand{\@tocsubsectionpnumformat}{#1}}
\newcommand{\setsubsubsectionpnumformat}[1]{\renewcommand{\@tocsubsubsectionpnumformat}{#1}}
\renewcommand{\@tocpagenum}[1]{%
  \hfill {\mdseries\@sectypepnumformat #1}}

\let\oldappendix\appendix
\renewcommand{\appendix}{%
  \leavevmode\oldappendix%
  \addtocontents{toc}{%
    \protect\settowidth{\protect\@tocsectionnumwidth}{\protect\@tocsectionformat\sectionname\space}%
    \protect\addtolength{\protect\@tocsectionnumwidth}{2em}}%
}
\makeatother

\makeatletter
\settocsectionnumwidth{2em}
\settocsubsectionnumwidth{2.5em}
\settocsubsubsectionnumwidth{3em}
\settocsectionindent{1pc}%
\settocsubsectionindent{\dimexpr\@tocsectionindent+\@tocsectionnumwidth}%
\settocsubsubsectionindent{\dimexpr\@tocsubsectionindent+\@tocsubsectionnumwidth}%
\makeatother

\settocsectionvskip{10pt}
\settocsubsectionvskip{0pt}
\settocsubsubsectionvskip{0pt}
    
\settocsectionformat{\bfseries}
\settocsubsectionformat{\mdseries}
\settocsubsubsectionformat{\mdseries}
\setsectionpnumformat{\bfseries}
\setsubsectionpnumformat{\mdseries}
\setsubsubsectionpnumformat{\mdseries}

\let\oldtableofcontents\tableofcontents
\renewcommand{\tableofcontents}{%
  \vspace*{-\linespacing}
  \oldtableofcontents}

\usepackage[bookmarks=true, bookmarksopen=true,%
bookmarksdepth=3,bookmarksopenlevel=2,%
colorlinks=true,%
linkcolor=blue,%
citecolor=blue,%
filecolor=blue,%
menucolor=blue,%
urlcolor=blue]{hyperref}

\usepackage{graphicx}
\usepackage{amssymb}
\usepackage{amsmath, amscd}
\usepackage[boxsize=5pt]{ytableau}

\newtheorem{thm}{Theorem}
\newtheorem{theorem}[thm]{Theorem}
\newtheorem*{theorem*}{Theorem}
\newtheorem{proposition}[thm]{Proposition}
\newtheorem{corollary}[thm]{Corollary}
\newtheorem{lemma}[thm]{Lemma}

\theoremstyle{definition}

\theoremstyle{remark}
\newtheorem{remark}[thm]{Remark}

\numberwithin{equation}{section}

\newcommand{\R}{{\mathbb{R}}}
\newcommand{\C}{{\mathbb{C}}}

\newcommand{\Z}{{\mathbb{Z}}}

\def\sP{{\mathsf{P}}}
\def\sW{{\mathsf{W}}}
\def\sZ{{\mathsf{Z}}}
\def\sA{{\mathsf{A}}}

\def\sPsi{{\mathsf{\Psi}}}
\def\be{\begin{equation}}
\def\ee{\end{equation}}

\newcommand{\Sk}{\operatorname{Sk}}

\usepackage{xspace}

\usepackage[margin=1in,marginparwidth=0.8in, marginparsep=0.1in]{geometry}

\usepackage{subcaption}

\begin{document}


\title{The skein valued mirror of the topological vertex}
\author{Tobias Ekholm}
\address{Department of Mathematics and Centre for Geometry and Physics Uppsala University, Box 480, 751 06 Uppsala, Sweden \and
Institut Mittag-Leffler, Aurav 17, 182 60 Djursholm, Sweden}
\email{tobias.ekholm@math.uu.se}
\author{Pietro Longhi}
\address{Department of Physics and Astronomy, Department of Mathematics, and Centre for Geometry and Physics, Uppsala University, Box 516, 751 20 Uppsala, Sweden}
\email{pietro.longhi@physics.uu.se}
\author{Vivek Shende}
\address{Center for Quantum Mathematics, Syddansk Univ., Campusvej 55
5230 Odense Denmark \and 
Department of mathematics, UC Berkeley, 970 Evans Hall,
Berkeley CA 94720 USA}
\email{vivek.vijay.shende@gmail.com}

\thanks{TE is supported by the Knut and Alice Wallenberg Foundation, KAW2020.0307 Wallenberg Scholar and by the Swedish Research Council, VR 2022-06593, Centre of Excellence in Geometry and Physics at Uppsala University and VR 2020-04535, project grant. \\ 
\indent  PL records the preprint number: UUITP-33/24.
\\
\indent  
VS is supported by  Villum Fonden Villum Investigator grant 37814, Novo Nordisk Foundation grant NNF20OC0066298, and Danish National Research Foundation grant DNRF157. 
}

\begin{abstract}
We count  holomorphic curves in complex 3-space
with boundaries on three special Lagrangian solid tori. In view of \cite{SoB}, the count is valued in the HOMFLYPT skein module of the union of the tori. Using 1-parameter families of curves at infinity, we derive three skein valued operator equations which must annihilate the count, and which dequantize to a mirror of the geometry.  We show algebraically
that the resulting equations determine the count uniquely, and that the result agrees with the topological vertex from topological string theory \cite{AKMV}.
\end{abstract}

\maketitle

\thispagestyle{empty}

\setcounter{tocdepth}{2}
\settocsectionvskip{6pt}
\tableofcontents

\newpage

\section{Introduction}\label{sec:intro}


The topological vertex is a combinatorial structure discovered in string theory \cite{AKMV}.  It takes as input three partitions $(\lambda_1,\lambda_2,\lambda_3)$ and returns a Laurent series in 
$q^{1/2}$ with integer coefficients.  One can take for its definition the skew Schur function formula \cite[Eq.~3.15]{Okounkov-Reshetikhin-Vafa}: 
\begin{equation}\label{vertex skew schur formula}
\mathcal{C}_{\lambda_1,\lambda_2,\lambda_3}(q) := q^{\kappa(\lambda_2)/2 + \kappa(\lambda_3)/2} \, s_{\lambda_2^t}(q^\rho) \, \sum_\eta s_{\lambda_3^t/\eta}(q^{\lambda_2 + \rho}) s_{\lambda_1/\eta}(q^{\lambda_2^t + \rho}),
\end{equation}
where $\rho = (-\frac12, -\frac32, \ldots)$ 
and $\kappa(\lambda)/2 = \sum_{(i, j) \in \lambda} (j-i)$.  The vertex has an elementary combinatorial interpretation, as an appropriately normalized count of box configurations with asymptotics $(\lambda_1,\lambda_2,\lambda_3)$ \cite{Okounkov-Reshetikhin-Vafa}, which in particular implies that it is invariant under cyclically permuting the partitions.  The formula \eqref{vertex skew schur formula} implies the symmetry 
$\mathcal{C}_{\lambda_1,\lambda_2,\lambda_3}= q^{\sum \kappa(\lambda_i)/2} \mathcal{C}_{\lambda_3^t, \lambda_2^t, \lambda_1^t}$.  
 The first few values are: 
$$\mathcal{C}_{\emptyset, \emptyset, \emptyset} = 1, \qquad \qquad \mathcal{C}_{\square, \emptyset, \emptyset} = (q^{1/2} - q^{-1/2})^{-1}, \qquad 
\qquad \mathcal{C}_{\square, \emptyset, \square}
= (q^{1/2} - q^{-1/2})^{-2} + 1.$$

The vertex was introduced to encode the partition function of topological open strings ending on a certain configuration of three special Lagrangian solid tori  in $\C^3$. It also provides a building block for similar calculations in general toric 3-folds \cite{AKMV}.  
In the mathematical literature, the topological vertex has been shown to compute appropriate relative versions of the Donaldson-Thomas \cite{MNOP} and  (closed) Gromov-Witten \cite{MOOP,MelissaandCo} invariants of $\mathbb{C}^3$.  

In this article, we will give a mathematical treatment in the original geometric framework: counting holomorphic curves in $\C^3$ with toric Lagrangian boundary conditions.  Such counts are well defined  after \cite{SoB, ghost, bare}. We will show geometrically, by counting holomorphic curves `at infinity' similarly to \cite{Ekholm:2020csl,ekholm2021coloredhomflypt,Scharitzer-Shende, Ekholm:2024ceb}, that the coefficients $\mathcal{T}_{\lambda_1, \lambda_2, \lambda_3}$ of the resulting partition function must satisfy the following system of equations 
(the first and second set of three are each related by cyclic permutation of all indices $1 \to 2 \to 3$):
\begin{align}
\label{eq:recursion-tem-by-term-split R32}\tag{$\mathrm{R}^3_2$}
	s_{\ydiagram{1}} (q^{\lambda_1 + \rho})  \,
	\mathcal{T}_{\lambda_1,\lambda_2,\lambda_3}
  =  
	\sum_{\beta_2 \in \lambda_2 - \ydiagram{1}}
	q^{(\kappa(\beta_2)-\kappa(\lambda_2))/2}\, \mathcal{T}_{\lambda_1,\beta_2,\lambda_3} - 
    	\sum_{\gamma_3 \in \lambda_3 + \ydiagram{1}}
  \mathcal{T}_{\lambda_1,\lambda_2,\gamma_3} ,
\\
\label{eq:recursion-tem-by-term-split R13}\tag{$\mathrm{R}^1_3$}
s_{\ydiagram{1}} (q^{\lambda_2 + \rho})
	\mathcal{T}_{\lambda_1,\lambda_2,\lambda_3} 
    =
	\sum_{\beta_3 \in \lambda_3 - \ydiagram{1}}
	q^{(\kappa(\beta_3)-\kappa(\lambda_3))/2}\, 
	\mathcal{T}_{\lambda_1,\lambda_2,\beta_3}
    - 	\sum_{\gamma_1 \in \lambda_1 + \ydiagram{1}}
	\mathcal{T}_{\gamma_1,\lambda_2,\lambda_3},
    \\
    \label{eq:recursion-tem-by-term-split R21}\tag{$\mathrm{R}^2_1$} 
    s_{\ydiagram{1}} (q^{\lambda_3 + \rho})
	\mathcal{T}_{\lambda_1,\lambda_2,\lambda_3}
    =  
	\sum_{\beta_1 \in \lambda_1 - \ydiagram{1}}
	q^{(\kappa(\beta_1)-\kappa(\lambda_1))/2}\, 
	\mathcal{T}_{\beta_1,\lambda_2,\lambda_3}
    - \sum_{\gamma_2 \in \lambda_2 + \ydiagram{1}}
	\mathcal{T}_{\lambda_1,\gamma_2,\lambda_3},
	\\
\label{eq:recursion-tem-by-term-split R31}\tag{$\mathrm{R}^3_1$}
s_{\ydiagram{1}} (q^{\lambda_2^t + \rho}) \,
\mathcal{T}_{\lambda_1,\lambda_2,\lambda_3}
  =
	\sum_{\alpha_1 \in \lambda_1 - \ydiagram{1}}
	\mathcal{T}_{\alpha_1,\lambda_2,\lambda_3}
    - \sum_{\gamma_3 \in \lambda_3 + \ydiagram{1}}
	q^{(\kappa(\gamma_3)-\kappa(\lambda_3))/2} \, 
	\mathcal{T}_{\lambda_1,\lambda_2,\gamma_3},	
        	\\
            \label{eq:recursion-tem-by-term-split R12}
\tag{$\mathrm{R}^1_2$}
s_{\ydiagram{1}} (q^{\lambda_3^t + \rho}) \,
\mathcal{T}_{\lambda_1,\lambda_2,\lambda_3}
	=  
	\sum_{\alpha_2 \in \lambda_2 - \ydiagram{1}}
	\mathcal{T}_{\lambda_1,\alpha_2,\lambda_3}
    -
	\sum_{\gamma_1 \in \lambda_1 + \ydiagram{1}}
	q^{(\kappa(\gamma_1)-\kappa(\lambda_1))/2} \, 
	\mathcal{T}_{\gamma_1,\lambda_2,\lambda_3},
\\
    \label{eq:recursion-tem-by-term-split R23}
    \tag{$\mathrm{R}^2_3$}
s_{\ydiagram{1}} (q^{\lambda_1^t + \rho}) \,
\mathcal{T}_{\lambda_1,\lambda_2,\lambda_3}
=
	\sum_{\alpha_3 \in \lambda_3 - \ydiagram{1}}
	\mathcal{T}_{\lambda_1,\lambda_2,\alpha_3}
    -
	\sum_{\gamma_2 \in \lambda_2 + \ydiagram{1}}
	q^{(\kappa(\gamma_2)-\kappa(\lambda_2))/2} \, 
	\mathcal{T}_{\lambda_1,\gamma_2,\lambda_3}.
\end{align}
By a manipulation of symmetric functions, we prove:

\begin{theorem} \label{vertex characterization}
The system of equations $(\mathrm{R}_i^j)$
with initial condition
$\mathcal{T}_{\emptyset, \emptyset, \emptyset}=1$ 
has unique solution 
$\mathcal{T}_{\lambda_1, \lambda_2, \lambda_3} \ = \ 
(-1)^{|\lambda_1| + |\lambda_2| + |\lambda_3|}\,\mathcal{C}_{\lambda_1^t, \lambda_2^t, \lambda_3^t}$.
\end{theorem}
This then constitutes the first mathematically rigorous proof that the topological vertex  indeed counts holomorphic curves with boundary on toric Lagrangians $L_1, L_2, L_3$ in $\C^3$.  

Let us review the geometric setup. 
We work in the formalism of skein-valued curve counting, which is a mathematically rigorous framework for counting holomorphic curves of all genera with Lagrangian boundary conditions $L$ in a Calabi-Yau 3-fold $X$ \cite{SoB, ghost, bare}.  Recall that the skein module of $L$ is generated by framed curves, subject to the relations of Figure \ref{HOMFLYPT skein}.  
The partition function $\sZ_{X, L}$ is an element of a certain completion of the HOMFLYPT skein module of $L$, and the count is defined by perturbing the holomorphic curve equation so that all relevant curves are embedded, and then counting each curve by the class determined by its boundaries in the skein, times a monomial factor in $a$ and $z$ recording topological information.  This works to define an invariant 
because the boundaries of moduli in 1-parameter families -- which ordinarily would obstruct invariance of curve counts -- are precisely canceled in the HOMFLYPT skein relations. 

\begin{figure}
    \includegraphics[scale=0.25]{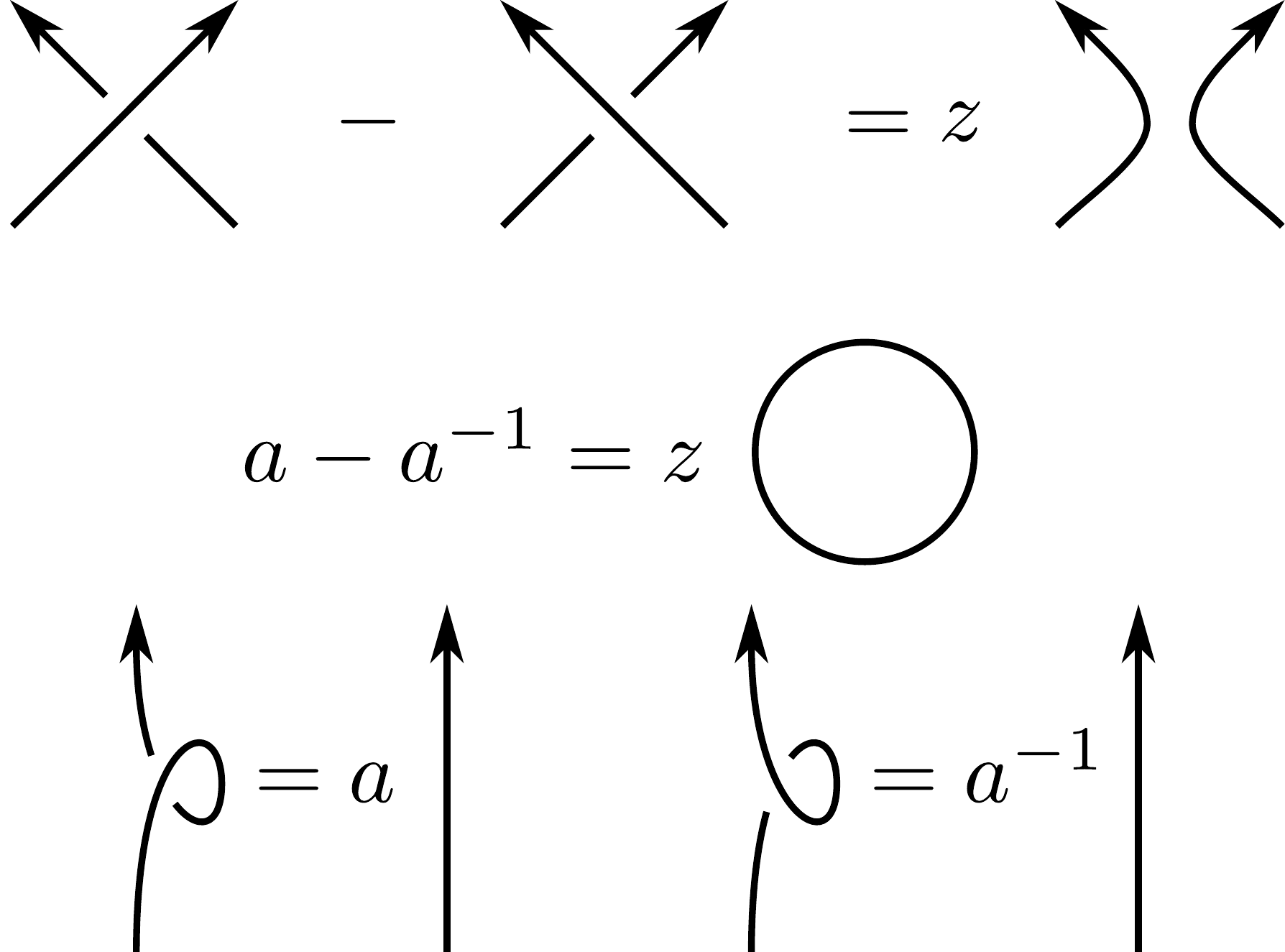}
    \caption{HOMFLYPT skein relations in the variable $z = q^{1/2} - q^{-1/2}$.} 
    \label{HOMFLYPT skein}
\end{figure}

Suppose $X$ is noncompact with convex contact boundary $\partial X$ and $L$ is asymptotic to a Legendrian $\partial L\subset\partial X$.  By a  `skein valued mirror', we mean some operators $\sA_{\partial X, \partial L} \subset \Sk( \R \times \partial L)$, which dequantize to give generators of the defining ideal for a moduli of objects in the (non-exact) partially wrapped Fukaya category of $(X, L)$, and which moreover satisfy 
\begin{equation} 
\label{what is SVM} \sA_{\partial X, \partial L} \cdot \sZ_{X, L} = 0.
\end{equation}
Here, $\Sk( \R \times \partial L)$ is an algebra by concatenation of links in the $\R$-direction, which, for the same reason, acts on $\Sk(L)$.  
This is consistent 
with the idea from topological strings that open holomorphic curve counts are `wave functions' for (i.e., are annihilated by) a quantization of the mirror family \cite{Aganagic:2000gs, ADKMV, AV2, Aganagic:2013jpa}.   In good cases, a skein-valued mirror can be constructed mathematically by studying certain boundary degenerations for 1-parameter families of curves `at infinity'
\cite{Ekholm:2020csl,ekholm2021coloredhomflypt,Scharitzer-Shende, Ekholm:2024ceb}; we will apply the same principles here. 

Let us recall facts about the relevant skeins.  The toric Lagrangians are topologically solid tori, $S^1 \times \R^2$ with ideal boundary a torus $T^2$, so our desired $\sA$ will be written in terms of elements of $\Sk(\R \times T^2)$.  We will write $\sP_{i,j} \in \Sk(\R \times T^2)$ for the class in the skein of the embedded curve of homology class $(i,j)$ in $0 \times T^2  \subset \R \times T^2$ framed by the vector field along the $\R$-factor.  In fact, these elements form a basis; the algebra structure is studied in detail in \cite{Morton-Samuelson}, where it is shown to be isomorphic with the $q=t$ specialization of the elliptic Hall algebra. 

We will be interested in the following elements of the tensor product of three skein modules of tori (the skein module of each torus gets its own $a$-variable, but they all have the same $q$-variable): 
\begin{align*}
\sA_1 \ =& \ 
a_1^{-1} \left(\sP_{0,0}^{(1)}  - \sP_{1,0}^{(1)} -  a_{1} a_3^2  \,\sP_{0,1}^{(1)}\right)
	+  a_2 a_3^2 
 \left(\sP_{0,0}^{(2)} + a_3^{-2}  \,\sP_{-1,1}^{(2)} - \sP_{-1,0}^{(2)}    \right)\\
	 & \ 
	+ \left(a_3  \,\sP_{0,0}^{(3)}  - \sP_{0,-1}^{(3)} + a_3 \sP_{1,-1}^{(3)}    \right)\\
\sA_2 \ =& \
a_2^{-1} \left(\sP_{0,0}^{(2)}  - \sP_{1,0}^{(2)} -  a_{2} a_1^2  \,\sP_{0,1}^{(2)}\right) +
	  a_3 a_1^2 
	\left(\sP_{0,0}^{(3)}   +  a_1^{-2}  \,\sP_{-1,1}^{(3)} - \sP_{-1,0}^{(3)}\right)\\
	& \ + 
	\left(a_1  \,\sP_{0,0}^{(1)} -\sP_{0,-1}^{(1)}  + a_1 \sP_{1,-1}^{(1)}   \right)\\
\sA_3  \ =& \
a_3^{-1} \left(\sP_{0,0}^{(3)}  - \sP_{1,0}^{(3)} -  a_{3} a_2^2  \,\sP_{0,1}^{(3)}\right) +
	  a_1 a_2^2 
\left(\sP_{0,0}^{(1)} +   a_2^{-2} \,\sP_{-1,1}^{(1)}
- \sP_{-1,0}^{(1)}  \right)\\
	 & \ +
	\left(  a_2  \,\sP_{0,0}^{(2)} - \sP_{0,-1}^{(2)}  + a_2 \sP_{1,-1}^{(2)} \right),
\end{align*}
Here, the superscript $(k)$ in $\sP_{i,j}^{(k)}$ indicates in which of the three tori the element lives. 

We will consider the action of the $\sA_i$ on the skein of the triple of solid tori in which the  $(1,0)$ classes become contractible.  To write formulas in coordinates, recall that  $\sP_{1,0}$ acts on $\Sk(S^1 \times \R^2)$
by adding at the `outside' the loop on the ideal boundary $T^2$
which is contractible in the solid torus.  In fact, this operator is diagonal in a basis $\{\sW_{\lambda,\bar\mu}\}$, indexed by pairs of partitions $(\lambda,\bar \mu)$, with distinct eigenvalues \cite{Hadji-Morton}.  This diagonalization determines $\sW_{\lambda,\bar\mu}$ up to scalar multiple. 
After choosing a longitude for the solid torus, we may define $\sW_{\lambda,\bar \mu}(\bigcirc)$ as 
the HOMFLYPT-polynomial for the $\sW_{\lambda,\bar\mu}$ satellite of the unknot in $S^3$. Then the aforementioned scalars are fixed by demanding that $\sW_{\lambda,\bar \mu}(\bigcirc)$ equals the quantum dimension of $(\lambda, \bar \mu)$ (expressed in terms of $(q^{1/2}, a)=(q^{1/2},q^{N/2})$).

A straightforward computation using the formulas of \cite{Morton-Samuelson} gives: 
\begin{proposition} \label{A to R intro}
    Let $\sZ = \sum \;\mathcal{T}_{\lambda_1, \lambda_2, \lambda_3} \sW_{\lambda_1,\emptyset} \otimes \sW_{\lambda_2,\emptyset} \otimes 
    \sW_{\lambda_3,\emptyset}$.  Then
    $\sA_1\cdot \sZ = 0$ if and only if $\mathcal{T}_{\lambda_1, \lambda_2, \lambda_3}$ satisfies $(\mathrm{R}^3_1)$ and $(\mathrm{R}^3_2)$, and similarly
    for cyclic permutations of indices $1 \to 2 \to 3$. 
\end{proposition}

Let $\sZ_{\C^3, L_1, L_2, L_3}$ be the skein-valued count of holomorphic curves in $\C^3$ ending on three toric Lagrangians, one along each `leg', see Section \ref{ssec : Lagrangians} and Figure \ref{fig:fillings}.
It is easy to show that the coefficients of $\sZ_{\C^3, L_1, L_2, L_3}$ in the basis 
$\sW_{\lambda_1,\bar \mu_1} \otimes \sW_{\lambda_2,\bar \mu_2} \otimes 
\sW_{\lambda_3,\bar \mu_3}$ vanish unless $\mu_1, \mu_2, \mu_3 = \emptyset$ and that the coefficient of $\sW_{\emptyset, \emptyset} \otimes \sW_{\emptyset, \emptyset} \otimes \sW_{\emptyset, \emptyset}$ is $1$.  
We count holomorphic curves to establish the main geometric result of the present article: 
\begin{theorem} \label{count at infinity intro} 
    $\sA_i\cdot \sZ_{\C^3, L_1, L_2, L_3} = 0$ for $i=1, 2, 3$.
\end{theorem}

We deduce formally from
Theorem \ref{vertex characterization}, Proposition \ref{A to R intro}, and Theorem \ref{count at infinity intro}: 

\begin{corollary} \label{TV counts intro} The topological vertex counts holomorphic curves with Lagrangian boundary: 
    \begin{equation}\label{eq : vertex counts curves}
    \sZ_{\C^3, L_1, L_2, L_3} \ = \ \sum\, (-1)^{|\lambda_1| + |\lambda_2| + |\lambda_3|}\,\mathcal{C}_{\lambda_1^t, \lambda_2^t, \lambda_3^t} \, \sW_{\lambda_1,\emptyset} \otimes \sW_{\lambda_2,\emptyset} \otimes 
    \sW_{\lambda_3,\emptyset}.
    \end{equation}
\qed
\end{corollary}
(The sign and transposes in the right hand side of \eqref{eq : vertex counts curves} could be eliminated by different choices of orientation and spin structures on the Lagrangians.)

\section{A recursion for the vertex (proof of Theorem \ref{vertex characterization})}
In this section we list straighforward lemmas on skew Schur functions and then use them to prove Theorem \ref{vertex characterization}.  

As always we regard Schur functions $s_\lambda$ and skew Schur functions $s_{\lambda/\nu}$ as functions of a large number of variables $\mathbf{x}=(x_1,\dots,x_N)$, where we take $N\to \infty$ in the following sense: an equation of Schur functions holds provided it holds for variables $\mathbf{x}$ with any finite number of components $N$.

\begin{lemma} \label{see saw sum} 
For any $\lambda, \mu$, we have:
    $$\sum_\eta \sum_{\tau \in \eta - \ydiagram{1}} s_{\lambda/\eta}(\mathbf{x}) s_{\mu / \tau}(\mathbf{y}) \ = \ \sum_\eta \sum_{\alpha \in \lambda - \ydiagram{1}} s_{\alpha/\eta}(\mathbf{x}) s_{\mu/\eta}(\mathbf{y}).$$
\end{lemma}
\begin{proof}
    Expanding in Schur functions, using Littlewood-Richardson coefficients to enforce $\tau \in \eta - \ydiagram{1}$ etc., and extracting the coefficient of $s_\delta(\mathbf{x}) s_{\epsilon}(\mathbf{y})$ shows the desired identity is equivalent to: 
    $$ 
    \sum_{\eta} \sum_{\tau} c^\lambda_{\eta \delta} c^{\eta}_{\tau \ydiagram{1}} c^{\mu}_{\tau \epsilon} \ = \ \sum_\eta \sum_{\alpha} c^{\lambda}_{\alpha \ydiagram{1}} c^\alpha_{\eta \delta} c^{\mu}_{\eta \epsilon}.
    $$
    To see that this identity holds, use the associativity relation
    $c^{\lambda}_{\alpha \ydiagram{1}} c^\alpha_{\eta \delta} = c^{\lambda}_{\alpha \delta} c^\alpha_{\eta \ydiagram{1}}$ on the right hand side and then relabel the summation variables $(\eta, \alpha) \mapsto (\tau, \eta)$.
\end{proof}

\begin{lemma} 
    For any $\lambda, \mu$, we have: 
    \begin{equation}
    \label{skew schur identity}
    \sum_{\eta} s_{\lambda/\eta}(\mathbf{x}) s_{\ydiagram{1}}(\mathbf{y}) s_{\mu/\eta}   (\mathbf{y})
    \ = \ 
    \sum_{\beta \in \mu+ \ydiagram{1}} \sum_\eta s_{\lambda/\eta}(\mathbf{x}) s_{\beta/\eta}(\mathbf{y}) -  \sum_{\alpha \in \lambda - \ydiagram{1}} \sum_{\eta} s_{\alpha/\eta}(\mathbf{x}) s_{\mu/\eta}(\mathbf{y}).\end{equation}
\end{lemma}
\begin{proof}
    According to the skew Pieri rule \cite{Assaf-McNamara}:  
    $$
    s_{\ydiagram{1}}(\mathbf{y}) s_{\mu/\eta}   (\mathbf{y})
    \ = \ 
    \sum_{\beta \in \mu+ \ydiagram{1}} s_{\beta/\eta}(\mathbf{y}) -  \sum_{\tau \in \eta - \ydiagram{1}}  s_{\mu/\tau}(\mathbf{y}).
    $$
    Multiply by $s_{\lambda/\eta}(\mathbf{x})$, 
    sum over $\eta$, and apply Lemma \ref{see saw sum}. 
\end{proof}

\begin{corollary}\label{C formula}
    The topological vertex
    \begin{equation} \label{vertex formula} \mathcal{C}_{\lambda_1,\lambda_2,\lambda_3} \ = \ q^{\kappa(\lambda_2)/2 + \kappa(\lambda_3)/2} \, s_{\lambda_2^t}(q^\rho) \, \sum_\eta s_{\lambda_3^t/\eta}(q^{\lambda_2 + \rho}) s_{\lambda_1/\eta}(q^{\lambda_2^t + \rho})
\end{equation}
satisfies the equation
$$
s_{\ydiagram{1}}(q^{\lambda_2^t + \rho})\, 
\mathcal{C}_{\lambda_1,\lambda_2,\lambda_3} \ = \ \sum_{\beta \in \lambda_1 + \ydiagram{1}} \mathcal{C}_{\beta,\lambda_2,\lambda_3}
- \sum_{\alpha \in \lambda_3 - \ydiagram{1}} q^{(\kappa(\lambda_3) - \kappa(\alpha))/2} \mathcal{C}_{\lambda_1, \lambda_2, \alpha},
$$
see \eqref{vertex skew schur formula} for notation.
\end{corollary}
\begin{proof}
     Specialize \eqref{skew schur identity}
    at $\lambda = \lambda_3^t$, $\mu= \lambda_1$, $\mathbf{y} = q^{\lambda_2^t + \rho}$, $\mathbf{x} = q^{\lambda_2 + \rho}$. 
\end{proof}

\begin{corollary} \label{R31 cor}
    The expression $\mathcal{T}_{\lambda_1, \lambda_2, \lambda_3}  := \, (-1)^{|\lambda_1| + |\lambda_2| + |\lambda_3|} \mathcal{C}_{\lambda_1^t, \lambda_2^t, \lambda_3^t}$  satisfies $(\mathrm{R}^3_1)$ in \linebreak 
    Theorem \ref{vertex characterization}: 
    $$s_{\ydiagram{1}} (q^{\lambda_2^t + \rho}) \,
\mathcal{T}_{\lambda_1,\lambda_2,\lambda_3}
  \ = \
	\sum_{\alpha_1 \in \lambda_1 - \ydiagram{1}}
	\mathcal{T}_{\alpha_1,\lambda_2,\lambda_3}
    - \sum_{\gamma_3 \in \lambda_3 + \ydiagram{1}}
	q^{(\kappa(\gamma_3)-\kappa(\lambda_3))/2} \, 
	\mathcal{T}_{\lambda_1,\lambda_2,\gamma_3}.$$	
\end{corollary}
\begin{proof}
    Apply Corollary \ref{C formula} with the partitions relabelled as $C_{\lambda_3, \lambda_2, \lambda_1}$.  Now multiply both sides by $q^{- \sum \kappa(\lambda_i) /2} (-1)^{ \sum |\lambda_i|}$.  The result
    shows that 
    $q^{- \sum \kappa(\lambda_i) /2} (-1)^{ \sum |\lambda_i|} C_{\lambda_3, \lambda_2, \lambda_1}$ satisfies $(\mathrm{R}_1^3)$.  We rearrange
    $q^{- \sum \kappa(\lambda_i) /2}  C_{\lambda_3, \lambda_2, \lambda_1} = 
     C_{\lambda_1^t, \lambda_2^t, \lambda_3^t}$. 
\end{proof}

Since symmetries of $\mathcal{T}_{\lambda_1,\lambda_2,\lambda_3}$ act transitively on the set of the $(\mathrm{R}_i^j)$, it follows that 
$\mathcal{T}_{\lambda_1,\lambda_2,\lambda_3}$ is annihilated
by these equations as well. 

\begin{proof}[Proof of Theorem \ref{vertex characterization}]
It remains to show that the $(\mathrm{R}_i^j)$ equations have a unique
solution $\mathcal{T}_{\emptyset, \emptyset, \emptyset}$ with $\mathcal{T}_{\emptyset,\emptyset,\emptyset} = 1$. 
First, let us show
$\mathcal{T}_{\emptyset,\emptyset,\lambda_3}$
is uniquely determined for all $\lambda_3$. 
The argument is inductive. 
Assume that we have determined $\mathcal{T}_{\emptyset,\emptyset,\lambda_3}$ for all partitions with at most $m$ parts.
Now if $\lambda_3$ has $m$ parts, 
the equation $(\mathrm{R}_2^3)$ for
$T_{\emptyset, \emptyset, \lambda_3}$ 
has only a single term with more than $m$ parts, namely $\mathcal{T}_{\emptyset,\emptyset,(\lambda_3, 1)}$, where $(\lambda_3, 1)$ is the partition by adding an extra part $1$ to $\lambda_3$.  Thus, $\mathcal{T}_{\emptyset,\emptyset,(\lambda_3, 1)}$ is also determined.

Consider now partitions with $m+1$ parts with $n>1$ as the smallest part, or with $m+2$ parts and $n-1,1$ as the smallest two parts, i.e.,  partitions of the form $(\lambda,n)$ and $(\lambda, n-1,1)$. 

The equations $(\mathrm{R}^3_1)$ and $(\mathrm{R}^3_2)$ for 
$\mathcal{T}_{\emptyset, \emptyset, (\lambda_3, n)}$ give a two-by-two system
\begin{equation}\label{eq : 2 by 2}
\left\{
\begin{matrix}
	&\mathcal{T}_{\emptyset,\emptyset,(\lambda_3,n+1)}  &+
	&\mathcal{T}_{\emptyset,\emptyset,(\lambda_3,n,1)}   &=  
    &F_{n},
	\\
	&q^{n}\, \mathcal{T}_{\emptyset,\emptyset,(\lambda_3,n+1)} 
	&+ &q^{-1}\mathcal{T}_{\emptyset,\emptyset,(\lambda_3,n,1)}   &= 
    &G_{n}
\end{matrix}
\right.,
\end{equation}
where $F_{n}$ and $G_{n}$ depend only on the $\mathcal{T}_{0, 0, (\lambda_3',n'+1)}$ and $\mathcal{T}_{(\lambda_3',n',1)}$, where either $|\lambda_3'| < |\lambda_3|$ or $n' < n$. 
Since 
\[
\det\left(\begin{matrix}1 & 1 \\ q^{n} & q^{-1} \end{matrix}\right) \ = \ (q^{-1} - q^{n}) \ \ne \ 0, 
\] 
the system has a unique solution, and recursively determines these coefficients. 

We next consider coefficients $\mathcal{T}_{\lambda_1,\emptyset,\lambda_3}$, where we go through partitions $\lambda_1$ as above using instead $(\mathrm{R}^1_3)$ and $(\mathrm{R}^1_2)$ which leads to the same two-by-two system with a more involved right hand side. To finish showing existence and uniqueness, we consider $\mathcal{T}_{\lambda_1,\lambda_2,\lambda_3}$ and go through partitions $\lambda_2$ using the corresponding  system of $(\mathrm{R}^2_1)$ and $(\mathrm{R}^2_3)$.
\end{proof}

\section{Recursion from operator equation (proof of Proposition \ref{A to R intro})}
In this section we prove Proposition \ref{A to R intro}. We will use the following notation:  
\begin{itemize}
    \item Let $\overline{\sW}_\lambda := \sW_{\lambda, \emptyset} \in \Sk(\R \times T^2)$.
    \item The content polynomial is
$C_\lambda(q) := \sum_{(i,j) \in \lambda} q^{j-i}$.  Note $C_\lambda(q) = C_{\lambda^t}(q^{-1})$, and the identity:
$$s_{\ydiagram{1}}(q^{\lambda+\rho}) = (q^{1/2}- q^{-1/2}) C_\lambda(q) + \frac{1}{q^{1/2} - q^{-1/2}}.$$
    \item Let $\lambda + \ydiagram{1}$ (resp. $\lambda - \ydiagram{1}$) denote the sets of 
partitions whose diagram is obtained by adding to (resp.~ removing from) $\lambda$ one box.  
    \item Let $\kappa(\lambda) := 2 \sum_{(i,j) \in \lambda} (j-i)$.   
\end{itemize}
Then, if $\beta \in \lambda + \ydiagram{1}$, we have $C_\beta(q) - C_\lambda(q) = q^{(\kappa(\beta)- \kappa(\lambda))/2}$.

We will also use the following formulas which are readily extracted from \cite{Morton-Samuelson}.  
Acting by the unknot:
\begin{align*}
	\sP_{0,0}\,\,
	\overline{\sW}_{\lambda} 
	& \ = \
	\frac{a - a^{-1}}{q^{1/2}-q^{-1/2}}\,
	\overline{\sW}_{\lambda} = 
    (a-a^{-1})s_{\ydiagram{1}}(q^\rho) \overline{\sW}_{\lambda}.
\end{align*}
Acting by meridian and longitude:
\begin{align*}
    \sP_{1,0}\,\,
	\overline{\sW}_{\lambda} 
	& \ = \ 
	\left( (a-a^{-1})s_{\ydiagram{1}}(q^\rho) + a(s_{\ydiagram{1}}(q^{\lambda+\rho}) - s_{\ydiagram{1}}(q^{\rho})) \right)
	\overline{\sW}_{\lambda}, 
	\\
	\sP_{0,1}\,\,
	\overline{\sW}_{\lambda} 
	& \ = \ 
	\sum_{\alpha \in \lambda + \ydiagram{1}}
	\overline{\sW}_{\alpha}.
\end{align*}
Acting by the negative meridian and a twisted longitude:
\begin{align*}
	\sP_{-1,0}\,
	\overline{\sW}_{\lambda} 
	& \ = \
	\left( (a-a^{-1})s_{\ydiagram{1}}(q^\rho) - a^{-1}(s_{\ydiagram{1}}(q^{\lambda^t+\rho}) - s_{\ydiagram{1}}(q^{\rho})) \right)
	\overline{\sW}_{\lambda}, 
	\\
	\sP_{-1,1}\,\,	\overline{\sW}_{\lambda} 
	& \ = \
	\sum_{\beta \in \lambda + \ydiagram{1}}
	a^{-1} (C_{\beta}(q^{-1}) - C_{\lambda}(q^{-1}))\,
	\overline{\sW}_{\beta},
	\\
    & \ = \ 
    \sum_{\beta \in \lambda + \ydiagram{1}}
	a^{-1} q^{(\kappa(\lambda) - \kappa(\beta))/2} \,
	\overline{\sW}_{\beta}. 
\end{align*}
Acting by the negative and a twisted negative longitude:
\begin{align*}    
	\sP_{0,-1}\,
	\overline{\sW}_{\lambda} 
	& \ = \
	\sum_{\gamma \in \lambda - \ydiagram{1}}
    \overline{\sW}_{\gamma} 
	+ \sW_{\lambda,\overline{\ydiagram{1}}}, 
	\\
	\sP_{1,-1}\,
	\overline{\sW}_{\lambda} 
	& \ = \ 
	- a \sum_{\gamma\in \lambda-\ydiagram{1}} 
	(C_{\gamma}(q) - C_{\lambda}(q))\, 
	\overline{\sW}_{\gamma} 
	+ a^{-1} 
 \sW_{\lambda,\overline{\ydiagram{1}}} 	
    \\ & \ = \ 	
    a \sum_{\gamma \in \lambda - \ydiagram{1}}
	 q^{(\kappa(\lambda)-\kappa(\gamma))/2}
    \overline{\sW}_{\gamma}
    + a^{-1} \sW_{\lambda,\overline{\ydiagram{1}}},
\end{align*}
We will in particular make use of the following operator combinations:  
\begin{align*}
    \left(\sP_{0,0}  - \sP_{1,0}\right)\,
	\overline{\sW}_{\lambda} 
	& \ = \ 
	- a(s_{\ydiagram{1}}(q^{\lambda+\rho}) - s_{\ydiagram{1}}(q^{\rho}))
    \,
	\overline{\sW}_{\lambda}, 
	\\
	\left(\sP_{-1,0}  - \sP_{0,0} \right)\,
	\overline{\sW}_{\lambda} 
	& \ = \
	- a^{-1}(s_{\ydiagram{1}}(q^{\lambda^t+\rho}) - s_{\ydiagram{1}}(q^{\rho})) \,
	\overline{\sW}_{\lambda}, 
	\\
	\left(\sP_{0,-1}  - a \sP_{1,-1}\right)\,
	\overline{\sW}_{\lambda} 
	& \ = \ 	
    \sum_{\gamma \in \lambda - \ydiagram{1}}
	\left(
	1 -
	a^2 q^{(\kappa(\lambda)-\kappa(\gamma))/2}
	\, 
	\right)
    \overline{\sW}_{\gamma}.
\end{align*}

Let
\[
\overline{\sW}_{\lambda_1\lambda_2\lambda_3} := \sW^{(1)}_{\lambda_1,\emptyset}\otimes\sW^{(2)}_{\lambda_2,\emptyset}\otimes \sW^{(3)}_{\lambda_1,\emptyset}
\]
and recall that 
\begin{align*}
\sA_1  \ &= \ a_1^{-1} \left(\sP_{0,0}^{(1)}  - \sP_{1,0}^{(1)} -  a_{1} a_3^2  \,\sP_{0,1}^{(1)}\right)\\
	&- \  a_2 a_3^2 
\left(\sP_{-1,0}^{(2)}  - \sP_{0,0}^{(2)} -  a_3^{-2}  \,\sP_{-1,1}^{(2)}\right)
	 - \ 
	\left(\sP_{0,-1}^{(3)}  - a_3 \sP_{1,-1}^{(3)} -  a_3  \,\sP_{0,0}^{(3)}\right).
\end{align*}

\begin{proof}[Proof of Proposition \ref{A to R intro}]
Suppose given
\be
	\sZ  \ = \ \sum_{\lambda_1,\lambda_2,\lambda_3} \, \mathcal{T}_{\lambda_1\lambda_2\lambda_3} \, \overline{\sW}_{\lambda_1\lambda_2\lambda_3}.	 
\ee
Then the condition that $\sA_1 \sZ =0$ is explicitly expressed in the coefficients
$\mathcal{T}_{\lambda_1, \lambda_2, \lambda_3}$ as: 
\begin{align}\label{eq:Z3-recursion-Tlmn}
&\left(
- (s_{\ydiagram{1}}(q^{\lambda_1+\rho}) - s_{\ydiagram{1}}(q^{\rho}))
+ a_3^2 
(s_{\ydiagram{1}}(q^{\lambda_2^t+\rho}) - s_{\ydiagram{1}}(q^{\rho}))
+(a_3^2 - 1) s_{\ydiagram{1}}(q^{\rho}) \right)
\mathcal{T}_{\lambda_1,\lambda_2,\lambda_3} 
	\\\notag
	&\quad\quad -  a_3^2 
	\sum_{\alpha_1 \in \lambda_1 - \ydiagram{1}}
	\mathcal{T}_{\alpha_1,\lambda_2,\lambda_3} 
	+  
	\sum_{\beta_2 \in \lambda_2 - \ydiagram{1}}
	q^{(\kappa(\beta_2) - \kappa(\lambda_2))/2} \, \mathcal{T}_{\lambda_1,\beta_2,\lambda_3}\\ \notag
	&\quad\quad -
	\sum_{\gamma_3 \in \lambda_3 + \ydiagram{1}}
	\left(
	1 -
	a_3^2 q^{(\kappa(\gamma_3) - \kappa(\lambda_3))/2}\, 
	\right)
	\mathcal{T}_{\lambda_1,\lambda_2,\gamma_3} \ = \ 0,
\end{align}	

We split the equation according to the $a_3$-degree to obtain the  
equations $(\mathrm{R}^3_1)$ and $(\mathrm{R}^3_2)$.  
As both the $\sA_i$ and $(\mathrm{R}_i^j)$ are cyclically symmetric 
for rotating indices $1 \to 2 \to 3$, the corresponding result for $i = 2, 3$ follows.
\end{proof}

\section{Operator equation from geometry  (proof of Theorem \ref{count at infinity intro})}\label{sec : operator equation}
In this section we first recall the general Symplectic Field Theory (SFT) approach to derive skein operator equations and then describe the contact geometry at infinity of the toric branes of the topological vertex and associated moduli spaces of holomorphic curves. Using this we prove Theorem \ref{count at infinity intro}.

\subsection{General strategy} \label{strategy}

Let $W$ be a symplectic Calabi-Yau 3-manifold with ideal convex boundary $\partial W$, and $L \subset W$ a Lagrangian asymptotic to a Legendrian $\partial L \subset \partial W$.  Let us recall from \cite{Ekholm:2018iso,Ekholm:2020csl,ekholm2021coloredhomflypt,Scharitzer-Shende, Ekholm:2024ceb} the method of using curve counts in the symplectization of the boundary  $(\R \times \partial W, \R \times \partial L)$ to determine curve counts in the interior $(W, L)$. 

As always in Floer theory, we obtain relations from the fact that the boundary of a  manifold -- here a moduli of holomorphic curves -- is zero in homology.  In particular, relations between numbers come from the boundaries of 1-dimensional moduli spaces.  Counting curves in the interior $(W,L)$ is an index zero problem, but by allowing (say) one positive puncture at an appropriate Reeb chord of $\partial L$, we find such 1-dimensional moduli.  The boundary of such a space can be identified with a collection of zero dimensional moduli spaces which, as they form a boundary, have total signed contribution zero to any curve count.  Some of these boundary terms -- the ones coming from boundary bubbling -- can be arranged in multiples of the skein relation (by the same argument which shows that the skein-valued curve count is invariantly defined) and so cancel in the skein valued curve counting.  

The remaining terms have to do with bubbling at infinity.  The SFT compactness and gluing theorems \cite{EGH,BEHWZ} (in fact we glue only at transverse Reeb chords, see \cite{EkholmRSFT,CELN}) tell us that such boundaries are indexed by pairs of: a curve in the symplectization  $(\R \times \partial W, \R \times \partial L)$ with the given positive punctures and arbitrary negative punctures, and a curve in the interior with positive punctures matching the aforementioned negative punctures.   There is also possibly breaking at orbits, but in the case at hand $\partial W = S^5$, there are no orbits of low enough index that this breaking can occur, so we omit it from the discussion.

As we have explained, when counted in the  skein, the sum of all such SFT boundary terms is zero.  To express this, we write $\Sk(\R \times \partial L; \mathrm{Reeb})$ to mean the skein of tangles possibly ending at the boundaries of Reeb chords on $\pm \infty \times \partial L$; similarly we write $\Sk(L, \mathrm{Reeb})$ to mean tangles that end at endpoints of Reeb chords in $\partial L$.  
We consider the operation 
$$\Sk(\R \times \partial L; \mathrm{Reeb}) \times \Sk(L, \mathrm{Reeb}) \to \Sk(L, \mathrm{Reeb})$$  which glues tangles if all endpoints match up, and returns zero otherwise. 

We write $\sA^-_\rho \in \Sk(\R \times \partial L; \mathrm{Reeb})$
for the skein-valued count of all rigid curves in the symplectization  with some fixed positive punctures $\rho$ and arbitrary negative punctures, and  $\sZ^+_{W, L} \in \Sk(L, \mathrm{Reeb})$ for the count of all rigid curves in $(W, L)$ with arbitrary positive punctures.
(For this to be well defined without further perturbation, we must assume no curve ever has multiple punctures going to the same Reeb chord; fortunately this will be the case in our application.) 

Then the SFT compactness and gluing gives: 
\begin{equation} \label{general equation}
\sA^-_\rho\cdot  \sZ^+_{W, L} = 0 \ \in \ \Sk(L, \mathrm{Reeb}).
\end{equation}

We are ultimately interested here not in 
$\sZ^+_{W, L}$, but rather in its summand consisting only of compact curves, $\sZ_{W, L} \in \Sk(L)$.   We would like to extract operators $\sA \in \Sk(\R \times \partial L)$ such that $\sA \cdot \sZ_{W, L} = 0$. 

In especially fortunate situations, when there are no Reeb chords of low enough index, one already has $\sZ_{W, L} = \sZ^+_{W, L}$, and to obtain the desired $\sA$ it only remains to choose some capping path for the positive end $\rho$.  For examples, see \cite{Ekholm:2020csl,Scharitzer-Shende}. 

We will not be so fortunate here.  More generally, one possibility would be to 
solve for the entire $\sZ^+_{W, L}$, and then take the compact curves at the end.  However, there is at least sometimes a shortcut:  some linear combination $\sum c_\rho \sA_\rho$ may give a tangle with no negative ends, which, therefore, necessarily annihilates $\sZ_{W, L}$.  Finding such linear combinations is at present an art rather than a science (we do not have a skein-valued elimination theory), but we have done this previously in simpler examples \cite{ekholm2021coloredhomflypt, Ekholm:2024ceb}, and will do so again here to obtain the desired operator equation.

\subsection{The toric branes of the topological vertex}\label{ssec : Lagrangians}
For $j=1,2,3$, fix positive real numbers $r_j>0$ and angles $\theta_j\in S^1 = \R/2\pi\Z$.
We will consider the following four configurations of Lagrangians of the topological vertex, parameterized by $(\rho_j,\alpha_j,\beta_j,)\in {\mathbb R}_+\times S^{1}\times S^{1}$,  see Figure \ref{fig:fillings}: 
on different legs
\begin{align*}
L_1 \ &: \ \left((\rho_1^2+r_1^2)^{1/2}e^{i\alpha_1}, \rho_1e^{i\beta_1},\rho_1e^{-i(\alpha_1+\beta_1-\frac{\pi}{2}+\theta_1)}\right),\\  
L_2 \ &: \ \left(\rho_2e^{i\alpha_2}, (\rho_2^2+r_2^2)^{1/2}e^{i\beta_2},\rho_2 e^{-i(\alpha_2+\beta_2-\frac{\pi}{2}+\theta_2)}\right),\\ 
L_3 \ &: \ \left(\rho_3e^{i\alpha_3}, \rho_3e^{i\beta_3},(\rho_3^2+r_3^2)^{1/2} e^{-i(\alpha_3+\beta_3-\frac{\pi}{2}+\theta_3)}\right);
\end{align*}
two on the same leg first cyclic order
\begin{align*}
L_1 \ &: \ \left((\rho_1^2+r_1^2)^{1/2}e^{i\alpha_1}, \rho_1e^{i\beta_1},\rho_1e^{-i(\alpha_1+\beta_1-\frac{\pi}{2}+\theta_1)}\right),\\  
L_2 \ &: \ \left((\rho_2^2+r_2^2)^{1/2}e^{i\alpha_2}, \rho_2e^{i\beta_2},\rho_2 e^{-i(\alpha_2+\beta_2-\frac{\pi}{2}+\theta_2)}\right),\\ 
L_3 \ &: \ \left(\rho_3e^{i\alpha_3}, (\rho_3^2+r_3^2)^{1/2}e^{i\beta_3}, \rho_3e^{-i(\alpha_3+\beta_3-\frac{\pi}{2}+\theta_3)}\right);
\end{align*}
two on the same leg second cyclic order
\begin{align*}
L_1 \ &: \ \left((\rho_1^2+r_1^2)^{1/2}e^{i\alpha_1}, \rho_1e^{i\beta_1},\rho_1e^{-i(\alpha_1+\beta_1-\frac{\pi}{2}+\theta_1)}\right),\\  
L_2 \ &: \ \left((\rho_2^2+r_2^2)^{1/2}e^{i\alpha_2}, \rho_2e^{i\beta_2},\rho_2 e^{-i(\alpha_2+\beta_2-\frac{\pi}{2}+\theta_2)}\right),\\ 
L_3 \ &: \ \left(\rho_3e^{i\alpha_3}, \rho_3e^{i\beta_3}, (\rho_3^2+r_3^2)^{1/2}e^{-i(\alpha_3+\beta_3-\frac{\pi}{2}+\theta_3)}\right);
\end{align*}
where in both cases we take $r_2<r_1$, and
all on the same leg with $r_3<r_2<r_1$,
\begin{align*}
L_1 \ &: \ \left((\rho_1^2+r_1^2)^{1/2}e^{i\alpha_1}, \rho_1e^{i\beta_1},\rho_1e^{-i(\alpha_1+\beta_1-\frac{\pi}{2}+\theta_1)}\right),\\  
L_2 \ &: \ \left((\rho_2^2+r_2^2)^{1/2}e^{i\alpha_2}, \rho_2e^{i\beta_2},\rho_2 e^{-i(\alpha_2+\beta_2-\frac{\pi}{2}+\theta_2)}\right),\\ 
L_3 \ &: \ \left((\rho_3^2+r_3^2)^{1/2}e^{i\alpha_3}, \rho_3e^{i\beta_3}, \rho_3e^{-i(\alpha_3+\beta_3-\frac{\pi}{2}+\theta_3)}\right).
\end{align*}

\begin{figure}
\centering
\begin{minipage}{0.15\textwidth}\centering
\includegraphics[width=1.0\textwidth]{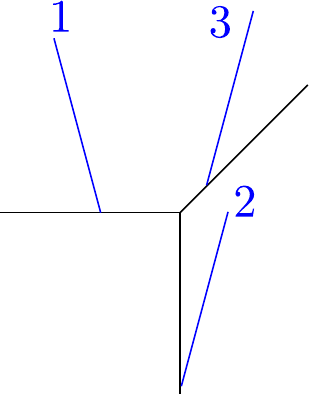}
\subcaption*{Filling 1}\label{fig:fill4}\end{minipage}
\hfill
\begin{minipage}{0.15\textwidth}\centering
\includegraphics[width=1.0\textwidth]{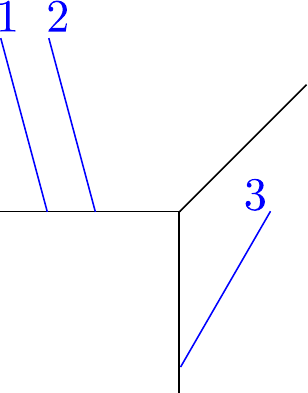}
\subcaption*{Filling 2}\label{fig:fill2}\end{minipage}
\hfill
\begin{minipage}{0.15\textwidth}\centering
\includegraphics[width=1.0\textwidth]{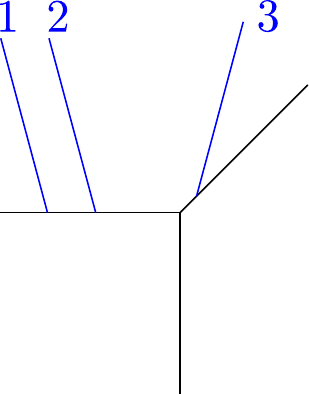}
\subcaption*{Filling 3}\label{fig:fill3}\end{minipage}
\hfill
\begin{minipage}{0.15\textwidth}\centering
\includegraphics[width=1.0\textwidth]{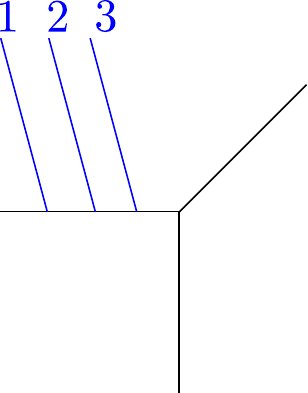}
\subcaption*{Filling 4}\label{fig:fill1}\end{minipage}
\caption{Fillings} \label{fig:fillings}
\end{figure}

We first note that for $\theta_j=0$, $j=1,2,3$, the Lagrangians are asymptotic to the same Legendrian in $S^5$,
\begin{equation*}
    \frac{1}{\sqrt{3}}\left(e^{i\alpha}, e^{i\beta},e^{-i(\alpha+\beta-\frac{\pi}{2})}\right), \quad\alpha,\beta\in S^1\times S^1,
\end{equation*}
with projection to $\mathbb{C}P^2$ that gives a threefold cover of the Clifford torus
\begin{equation*}
   \left[e^{i\phi} : e^{i\theta} : 1\right],\quad \phi,\theta\in S^1. 
\end{equation*}

Note then that increasing $\theta_j$ from $0$ to $t$ corresponds geometrically to a shifting of $\Lambda_j$ $t$ units along the Reeb flow in $S^5$. We will use a symmetric configuration of the three branes and take $\Lambda_1$ with $\theta_1=0$, $\Lambda_2$ with $\theta_2=\frac{2\pi}{9}$, and $\Lambda_3$ with $\theta_3=\frac{4\pi}{9}$. Noting that a shift by $\frac{2\pi}{3}$ takes $\Lambda_j$ to itself, we write $\Lambda_k$ for the shift by $\frac{2\pi k}{9}$. 

\subsection{Leading disk and annulus contributions}
We will determine the details of the skein operator equation using information about the simplest compact curves with boundary on one or two toric branes on different legs. These curves are as follows.
\begin{lemma}\label{l : simple disks and annuli}
$\quad$
\begin{itemize}
\item[$(a)$]
There is a single holomorphic disk on $L_k$ that lies in the $k^{\rm th}$ coordinate line and has boundary along the curve given by $\rho_k=0$. 
\item[$(b)$]
The nodal annuls corresponding to the union of any two such disks contributes a single holomorphic annulus stretching between branes. 
\end{itemize}
\end{lemma}

\begin{proof}
The coordinate disks are easy to find and it easy to check that they are transversely cut out, this proves $(a)$. Further, it is straightforward to check that there is exactly one annulus stretching between two of the branes, it is nodal and has image the union of the two disks. We need to check that it contributes by one annulus to the partition function. This requires studying the variation of the nodal annulus in the space of annuli. Using the parametrization of such a neighborhood as in \cite{HWZGWpolyfolds} this amounts to proving that the appropriate linearized operator is surjective, which technically is the same a proving a Floer gluing theorem for gluing the disk on $L_1$, considered as stationary, to the family of disks on $L_2$ obtained by translating $L_2$ in the $z_3$-plane. Such gluing arguments are standard and gives a family of transveresly cut out annuli over a punctured neighborhood $U$ of the origin in the $z_3$-plane, compactified by the nodal annuli between $L_1$ and the original $L_2$.     
\end{proof}

\begin{remark}
We point out that the count of annuli in Lemma \ref{l : simple disks and annuli} could also be carried out by observing that it is given by the extremal part of the HOMFLYPT polynomial of the Hopf link in $S^3$, by \cite{SoB, Ooguri:1999bv}.
\end{remark}

\subsection{Reeb chords and moduli spaces at infinity}
Reeb chords of $\Lambda=\Lambda_1\cup\Lambda_2\cup\Lambda_3$  come in Bott families, where the families are copies of the torus components of the link.  The following are the Reeb chord families of minimal action:
\begin{itemize}
    \item The families $B(a_{k,k+1})$ of action $\frac{2\pi}{9}$ chords $\Lambda_k\to\Lambda_{k+1}$   with minimum chord $a_{k,k+1}$.
    \item The families $B(b_{k,k+2})$ of action $\frac{4\pi}{9}$ chords $\Lambda_k\to\Lambda_{k+2}$   with minimum chord $b_{k,k+2}$.
    \item The families $B(c_{k,k})$ of action $\frac{2\pi}{3}$ chords $\Lambda_k\to\Lambda_k$ with minimum chord $c_{k,k}$.
\end{itemize}
We think of the symplectization of $S^5$ as $\C^3\setminus\{0\}$ and use the holomorphic projection $\C^3\setminus\{0\}\to \C P^2$. As usual in symplectizations, we use a complex structure invariant under $\R$-translation (corresponds to scaling by $\R_+$ in $\C^3$) and holomorphic curves will therefore always come in $\R$-families. We say that a curve is rigid if its moduli space is transversely cut out and canonically isomorphic to $\R$ via translations. We say that a curve with one positive puncture is \emph{Bott rigid} if it is rigid after fixing the positive puncture or equivalently if the tangent space is Bott transverse and equal to that of the Bott manifold at the positive puncture.  
\begin{lemma}\label{l : Bott curves}
Moduli spaces of Bott-rigid holomorphic curves in $\C^3$ with boundary on $\Lambda\times\R$ and positive puncture at a chord of action $\le\frac{2\pi}{3}$ are the following: 
\begin{itemize}
\item[$(i)$] Three families of disks with positive puncture at $c\in B(c_{k,k})$ and no other punctures.  The boundaries of these disks project to disjoint embedded paths in $\Lambda$ between the endpoints of the chord $c_{k,k}$. 
\item[$(ii)$] Two families of triangles with positive puncture at $c\in B(c_{k,k})$ and negative punctures, at $a\in B(a_{k,k+1})$ and $b\in B(b_{k+1,k+3})$, respectively, at $b\in B(b_{k,k+2})$ and $a\in B(a_{k+2,k+3})$. These triangles are contained in the fiber of the Hopf fibration.
\item[$(iii)$] One family of triangles with positive puncture at $b\in B(b_{k,k+2})$ and negative punctures at $a\in B(a_{k,k+1})$ and $a'\in B(a_{k+1,k+2})$. These triangles are contained in the fiber of the Hopf fibration. 
\end{itemize}
\end{lemma}

\begin{proof}
Note that there is a natural 1-1 correspondence between holomorphic curves in $\C^3\setminus\{0\}$ with boundary on $\Lambda$ and curves in $\C P^2$ with boundary on the Clifford torus together with a meromorphic section of the Hopf bundle. Holomorphic curves with boundary on the Clifford torus of minimal area are well known and give the first families. Explicitly, the three disk families are given by maps $u_j\colon D\to \C^3\setminus\{0\}$, where $D$ denotes the unit disk in the complex plane as follows:
\begin{align*}
u_1(z) &= (e^{i\alpha}z-1)^{-\frac{2}{3}}\left(e^{i\alpha}z,e^{i\beta},ie^{-i\beta}\right),\quad \alpha,\beta\in[0,2\pi);\\
u_2(z) &= (e^{i\beta}z-1)^{-\frac{2}{3}}\left(e^{i\alpha},e^{i\beta}z,ie^{-i\alpha}\right),\quad \alpha,\beta\in[0,2\pi);\\
u_3(z) &= (e^{i\gamma}z-1)^{-\frac{2}{3}}\left(e^{i\alpha},e^{-i\alpha},ie^{i\gamma}z\right),\quad \alpha,\gamma\in[0,2\pi).
\end{align*}
For a fixed Reeb chord, starting at $(1,1,1)$, say, the boundaries of the three holomorphic disks projected to the Legendrian torus $\Lambda$ are the following curves: 
\begin{align}\label{eq : disk j}
p_1 \;:\; u_1(e^{i\theta}) &= \left(e^{i\frac{2\theta}{3}},e^{-i\frac{\theta}{3}},ie^{-i\frac{\theta}{3}}\right),\quad \theta\in(0,2\pi);\\\notag
p_2 \; : \; u_2(e^{i\theta}) &= \left(e^{-i\frac{\theta}{3}},e^{i\frac{2\theta}{3}},ie^{-i\frac{\theta}{3}}\right),\quad \theta\in(0,2\pi);\\\notag
p_3 \; : \;u_3(e^{i\theta}) &= \left(e^{-i\frac{\theta}{3}},e^{-i\frac{\theta}{3}},ie^{i\frac{2\theta}{3}}\right),\quad \theta\in(0,2\pi),
\end{align}
see Figure \ref{fig:reebchordendpoints}.

\begin{figure}
    \centering
    \includegraphics[width=0.3\linewidth]{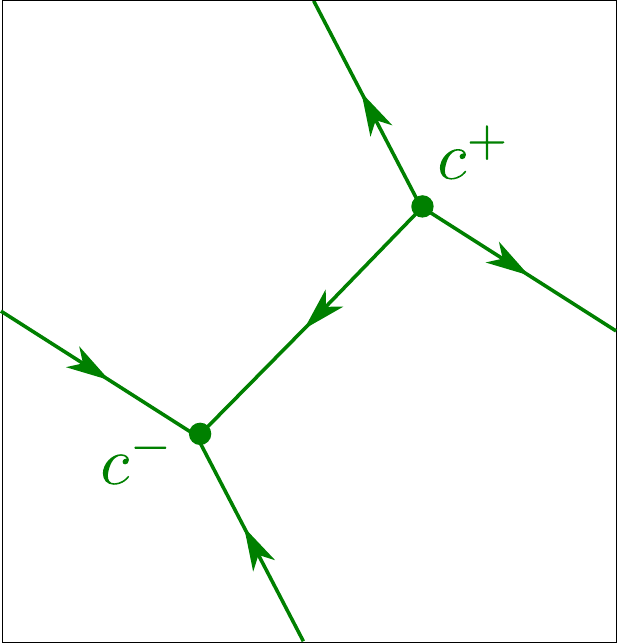}
    \caption{The boundaries of  three disks on the torus at infinity.}
    \label{fig:reebchordendpoints}
\end{figure}


The curves in families $(i)$ and $(ii)$ have zero projected area and must therefore be contained in the fiber. Each fiber is a copy of the complex plane $\C$ and the intersection of the fiber with the three Lagrangians are rotations of the following configuration of rays:
\begin{align*}
\Lambda_1\times \R \cap \C^\ast &= \R_+\cup e^{i\frac{2\pi}{3}}\R_+\cup e^{i\frac{4\pi}{3}}\R_+,\\   
\Lambda_2\times \R \cap \C^\ast &= e^{i\frac{2\pi}{9}}\R_+\cup e^{i\frac{8\pi}{9}}\R_+\cup e^{i\frac{14\pi}{9}}\R_+,\\
\Lambda_3\times \R \cap \C^\ast &= e^{i\frac{4\pi}{9}}\R_+\cup e^{i\frac{10\pi}{9}}\R_+\cup e^{i\frac{16\pi}{9}}\R_+,
\end{align*}
It is then elementary to determine the curves in $(ii)$ and $(iii)$. We illustrate by two examples. A curve in $(ii)$ is given by the sector bounded by $\R_+$ and $e^{i\frac{2\pi}{3}}\R_+$ with a slit in the ray $e^{i\frac{2\pi}{9}}\R_+$ or with a slit in the ray $e^{i\frac{4\pi}{9}}\R_+$. A curve in $(iii)$ is given by the sector bounded by $\R_+$ and $e^{i\frac{4\pi}{9}}$ with a slit in the ray $e^{i\frac{2\pi}{9}}\R_+$. 
\end{proof}

\subsection{Framing, $4$-chain, and spin structure choices} \label{framing spin choices}
In the skein-valued curve counting of \cite{SoB}, we equip each Lagrangian $L$ as always with a spin structure $\sigma$, but also with the data of a vector field $v$ and 4-chain $W$, with the properties that $\partial W  = 2L$ and $W = \pm Jv$ near $L$.  In the noncompact case, we ask that $W, v$ are (asymptotically) scale invariant near infinity.  Here we fix these choices.

We construct the $4$-chain $W_1$ of $L_1$ as follows. Foliate $L_1$ by `meridional circles' $S_{(\rho_1,\alpha)}$:  
\[
S_{(\rho_1,\alpha)} = \left((r_1^2+\rho_1^2)^{1/2}e^{i\alpha},\rho_1e^{i\beta},\rho_1e^{-i(\alpha+\beta-\frac{\pi}{2})}\right), \quad 0\le \beta<2\pi.
\] 
Then $S_{(\rho_1,\alpha)}$ lies in the coordinate plane $\approx\C^2$ of the last two coordinates given by the equation $z_1=(\rho_1^2+r_1^2)^{1/2}e^{i\alpha}$, inside a Lagrangian plane $\Pi_{(\rho_1,\alpha)}\subset\C^2$ given by the equation $iz_3=e^{i\alpha}\overline{z_2}$. Let $w_{(\rho_1,\alpha)}$ denote a unit vector field in $\C^2$ perpendicular to $\Pi_{(\rho_1,\alpha)}$. Then $w_{(\rho_1,\alpha)}$ gives a vector field in an $SO(2)$-bundle over $S^1\times\R_+$ with fiber at $(\alpha,\rho_1)$ the orthogonal complement to $\Pi_{(\rho_1,\alpha)}\subset\C^2$. Our $4$-chain will depend on the choice of vector field $w_1=w_{(\rho_1,\alpha)}$ and we denote it $W_1(w_1)$ to indicate this dependence.

Let $\Gamma_{(\rho_1,\alpha)}^\pm$ denote the cylinder that is the union of length $\frac12\rho_1$ intervals over $S_{(\rho_1,\alpha)}$ along the line in direction $i\cdot\nu$, where $\nu$ is the normal to the circle in the plane. For concreteness, if we identify the  Lagrangian plane with $\R^2\subset \C^2$ with the radius $\rho_1$ circle 
\[
S_{(\rho_1,\alpha)}\approx(x_1,y_1,x_2,y_2)=(\rho_1\cos\beta,0,\rho_1 \sin\beta,0),\quad 0\le \beta<2\pi,
\]
we have
\[
\Gamma_{(\rho_1,\alpha)}^{\pm}\approx (\rho_1\cos\beta,\pm s\rho_1\cos\beta,\rho_1 \sin\beta,\pm s\rho_1 \sin\beta),\quad 0\le \beta<2\pi,\; 0\le s\le \tfrac12. 
\]
Let $\phi\colon [0,1]\to[0,1]$ be an approximation of the identity function equal to $0$ in a neighborhood of $0$ and equal to $1$ in a neighborhood of $1$ and take $D_{(\rho_1,\alpha)}^\pm$ be the smooth disk through $\pm\rho_1w_{(\rho_1,\alpha)}$ given in coordinates by
\begin{align*}
D_{(\rho_1,\alpha)}^\pm \ =& \ (\sigma\rho_1\cos\beta, 0,\sigma\rho_1 \sin\beta,0) \\ 
+& \ \phi(\sigma)\left(0,\pm \tfrac12\rho_1\cos\beta,0,\pm \tfrac12\rho_1 \sin\beta)
 + \ \pm(1-\phi(\sigma)\right)\rho_1 w,\\
 & \ 0\le\beta<2\pi,\;0\le \sigma\le 1,
\end{align*}
where the unit vector $w$ corresponds to $w_{(\rho_1,\alpha)}$.
Define, $\Sigma^{\pm}_{(\rho_1,\alpha)}=\Gamma_{(\rho_1,\alpha)}^{\pm}\cup D_{(\rho_1,\alpha)}^\pm$ and take
the $4$-chain $W_1(w_1)$ as
\[
W_1(w_1)=\bigcup_{(\rho_1,\alpha_1)\in \R_+\times S^1}\left(\Sigma_{(\rho_1,\alpha)}^{+}\cup \Sigma_{(\rho_1,\alpha)}^{-}\right).
\]
Then $\partial W_1(w_1)=2\cdot L_1$.

We define $W_2(w_2)$ and $W_3(w_3)$ in the same way using foliations by meridional circles in the coordinate planes of the first and last and the two first coordinates, respectively.
\begin{lemma}\label{l : first properties of 4-chain}
Let $\nu$ denote the inward normal of the $4$-chains $W_k(w_k)$. Then the following holds.  
\begin{itemize}
\item[$(a)$] Outside a neighborhood of the central circle, $\rho_k=0$, $J\nu$ is a radial vector field and thus meets the required conditions at infinity.
\item[$(b)$] Along the central circle $\rho_k=0$, $J\nu$ agrees with the vector field $J w_{k}$.
\item[$(c)$] If $D_k$ denotes the basic holomorphic disks on $L_k$ in Lemma \ref{l : simple disks and annuli} then $W_k\cap \mathrm{int}(D_j)=\emptyset$ for all $k,j$.
\end{itemize}
\end{lemma}
\begin{proof}
Properties $(a)$ and $(b)$ are evident by construction. For $(c)$, note that the $4$-chain $W_k(w_k)$ is contained in the region  $\{|z_k|\ge r_k\}$ and that all basic disks have interiors in the complement to that region. 
\end{proof}

Consider the boundaries $p_k$ of the disks in \eqref{eq : disk j}. Here the curves $p_{k+1}-p_{k+2}$ and $p_{k}-p_{k+2}$ gives two closed curves on the Legendrian $\Lambda_k$. If we equip these curves with the framing given by the outward normal vector field they give elements $\sP_{1,0}$ and $\sP_{0,1}$ in the skein of $\Lambda_k$. 

Let $v_k=J(\nu)$, where $\nu$ is the inward normal to $W_k(w_k)$ then flowing the central circle $\rho_k=0$ in $L_k$ toward infinity along $v_k$ we get a curve $l_k(w_k)$ in $\Lambda_k$ that is homologous to $f\sP_{1,0}+\sP_{0,1}$ for some $f\in\Z$, $[l_k(w_k)]=f[\sP_{1,0}]+[\sP_{0,1}]$.  
\begin{lemma}\label{l : vector filed for 4-chain}
There is a unique homotopy class of vector fields $w_k$ such $f=0$, i.e., such that
\[
[l_k(w_k)]=0\cdot[\sP_{1,0}]+[\sP_{0,1}].
\]
\end{lemma}
\begin{proof}
Adding a local twist in $w_k$ changes the homotopy class of $[l_k(w_k)]$ by $\pm[P_{1,0}]$. 
\end{proof}

We define the $4$-chain of $L_k$ as $W_k=W_k(w_k)$, where $w_k$ is a vector field for which Lemma~\ref{l : vector filed for 4-chain} holds.

For the spin structure, on the toric branes $L_1, L_2, L_3$, there are on each two choices of spin structure.  The choice of spin structure will be reflected in the sign of the disks $D_k$ of Lemma~\ref{l : simple disks and annuli}.  We choose the spin structure which gives the sign `minus'. 
 
Recall that the definition of the $\sW_{\lambda, \overline{\mu}}$ basis of a torus depends on the choice of longitude and meridian of the boundary torus, with the meridian of course being already determined up to orientation. We pick the meridian $\sP_{1,0}$ and longitude $\sP_{0,1}$ adapted to the disk boundaries as above. Then by choice of $w_k$, $\sW_{\ydiagram{1}, \emptyset}$ in $L_k$ agrees with
$\partial D_k \in \Sk(L_k)$, framed by $v_k=Jw_k$ since it is framed isotopic to the longiutude $\sP_{1,0}$ with framing given by the outward vector field.

\subsection{Morsification}
We will perturb the contact form in order to Morsify the Reeb chords of action $<\frac{8\pi}{9}$ discussed above.  Consider the shift $\Lambda_k^\epsilon$ of $\Lambda_k$ by $\epsilon$ units along the Reeb flow and a neighborhood of $\Lambda_k^\epsilon$ of the form
\[
[-\delta,\delta]\times D_r T^\ast\Lambda_k^\epsilon,
\]
where $D_r$ denotes the radius $r$ disk bundle,  
where $\alpha=dz-pdq$, and where $z$ is a coordinate along the interval and $pdq$ is the canonical $1$-form on $T^\ast\Lambda_k^\epsilon$. We perturb the contact form $\alpha$ to $e^{h}\alpha$ where $h\colon S^5\to\R$ is supported in small disjoint neighborhoods of $\Lambda_k^\epsilon$ of the above form. In other words, $h=h_1+h_2+h_3$ where $h_k$ is supported in a small neighborhood of $\Lambda_k^\epsilon$. Here we take it to have the form
\[
h_k(q,p,z)= \beta_1(z)\beta_2(|p|)f_k(q), 
\] 
where $f_k\colon\Lambda_k^\epsilon\to\R$ is a Morse function, $\beta_1(z)$ is a cut off function equal to $1$ near $z=0$ and equal to $0$ for $|z|\ge \delta'$ for some small $\delta'<\delta$, and where $\beta_2(|p|)=1$ near $p=0$ and equal to $0$ for $|p|<\delta''$. Then the Reeb vector field $R_h$ of $e^h\alpha$ is given by
\[
R_h= e^{-h}(R-X_h),
\]   
where $X_h$ is the $d\alpha$-dual (Hamiltonian vector field) of the restriction of $dh$ to the contact plane.

Note then that the map from $\Lambda_k$ to $T^\ast\Lambda_{k+j}$, $j=1,2,3$, obtained by flowing time $\frac{2j\pi}{9}$ and then projecting to $T^\ast\Lambda_{k+j}$ agrees with a small time Hamiltonian flow of the sum of $j$ of the functions $f_k(q)$. In particular, Reeb chords correspond to the critical points of the sum of functions $f_k$. We will use the Reeb chords at the minimum of these Morse functions and by adjusting the Morse function, we are free to choose this point. In order to use the disks in \eqref{eq : disk j}, we take the minimum near $(1,1,1)$.  

We will use the correspondence between holomorphic curves near the Bott degenerate limit and holomorphic curves in the Bott degenerate limit with Morse flow lines attached. The curves with Bott-degenerate negative asymptotics considered here correspond, after projecting out the symplectization direction, to flow trees with a positive puncture at a minimum and two negative punctures also at minima. The holomorphic curve/Morse flow tree correspondence result needed is then a very simple instance of \cite[Theorem 1.1]{Ekholmflowtree}. (In the language of \cite{Ekholmflowtree}, the front projection of the Legendrian is smooth and there are only three local sheets, so the trees have only a single $Y_0$-vertex). 

Thus the element $\sA_k^- := \sA_{c_{k,k}}^- \in \Sk(\R \times (\Lambda_1 \sqcup \Lambda_2 \sqcup \Lambda_3), \mathrm{Reeb})$ is a sum of 5 terms, corresponding to the three disks mentioned in $(i)$ and the two triangles mentioned in $(ii)$. Note $\sA_k^-$ implicitly depends on a choice of Morse function to perturb out of the Bott situation, and also on choices of brane data (spin structures, vector fields, 4-chains) for the $\R \times \Lambda_i$.

\begin{lemma} \label{triangles opposite signs}
    For any of the $\sA^-_k$, the contributions of the two triangles 
    $\Delta(a_{k,k+1}b_{k+1,k+3})$ and $\Delta(b_{k,k+2},a_{k+2,k+3})$ from (ii)
    have opposite signs. Also, if $\alpha^{k}_{k+j}$, $j=1,2$ denote the $(a_1,a_2,a_3)$-monomials of the triangles $\Delta(a_{k,k+1}b_{k+1,k+3})$, $\Delta(b_{k,k+2},a{k+2,k+3})$ with boundary on $\Lambda_k$ and $\Lambda_{k+j}$ then 
    \begin{equation}\label{eq : a powers}
    \alpha^k_{k+1}\,\alpha^{k+1}_{k+2}\,\alpha^{k+2}_{k} \ = \ \alpha^{k}_{k+2}\,\alpha^{k+1}_{k}\,\alpha^{k+2}_{k+1}.
    \end{equation}
\end{lemma}
\begin{proof}
Consider the moduli space of disks with positive puncture at $c_{k,k}$ and three negative punctures, at $a_{k,k+1}$, $a_{k+1,k+2}$, and $a_{k+2,k+3}$. The boundary of this $1$-dimensional moduli space consists of two two-level curve configurations with the following parts:
\begin{align*}
&\Delta(a_{k,k+1}b_{k+1,k+3}) \quad\text{glued at $b_{k+1,k+3}$ to}\quad \Delta(a_{k+1,k+2}a_{k+2,k+3}),\\
&\Delta(b_{k,k+2}a_{k+2,k+3}) \quad\text{glued at $b_{k,k+2}$ to}\quad \Delta(a_{k,k+1}a_{k+1,k+2}).
\end{align*}
Now, the triangles $\Delta(a_{k+1,k+2}a_{k+2,k+3})$ with positive puncture at $b_{k+1,k+3}$ and $\Delta(a_{k,k+1}a_{k+1,k+2})$ with positive puncture at $b_{k,k+2}$ are related by the $\mathbb{Z}_3$ symmetry of the vertex and hence have the same sign. Since the orientation signs at opposite ends of an oriented $1$-manifold are opposite, it follows that the signs of $\Delta(a_{k,k+1}b_{k+1,k+3})$ and $\Delta(b_{k,k+2},a_{k+2,k+3})$ are opposite.

We next consider $a_k$-powers. Let $\alpha^{k}_{k+j}$, $j=1,2$ denote the $(a_1,a_2,a_3)$-monomials of the triangles $\Delta(a_{k,k+1}b_{k+1,k+3})$, $\Delta(b_{k,k+2},a{k+2,k+3})$ with boundary on $\Lambda_k$ and $\Lambda_{k+j}$ and, similarly, let $\beta^{k}$ be the monomial of the triangle $\Delta(a_{k,k+1}a_{k+1,k+2})$. Since the the $(a_1,a_2,a_3)$-monomial of a holomorphic building is the product of the monomials of its parts we find the relations
\[
\alpha^k_{k+2}\beta^k=\alpha^k_{k+1}\beta^{k+1},
\]
which implies \eqref{eq : a powers}. The lemma follows. 
\end{proof}

\begin{proposition} \label{cancelling triangles}
    Fix some choice $p_j$ with $j = 1,2,3$,
    and use it as a capping path for all $c_{k,k}$. Then for an appropriate choice of Morse function, there is a choice of signs and framing monomomials $\ast$ so that all contributions of triangles to 
    $$\sA_{j} := \pm \ast \sA_1^- \pm \ast\,\, \sA_2^- \pm \ast\,\, \sA_3^-$$
    will cancel. 
\end{proposition}
\begin{proof}
    Note that in the above sum, there are six triangles, which by Lemma \ref{triangles opposite signs} may be re-organized into pairs with matching negative punctures.  Note the members of the pair come from different $\sA_j^-$.  Thus it is at least plausible that for the right choice of capping path for the positive puncture, and Morse function perturbation (which influences the triangle boundary), the paths traced by the boundary of these two triangles will match.  
    
    Suppose this has been accomplished.  Then some appropriate $\sA_1 \pm \ast \sA_2$ will cancel the common triangle.  It follows from Lemma \ref{triangles opposite signs} that the two remaining triangles in this sum still contribute with opposite signs.  But these are the same triangles which appear in $\sA_3$, also with opposite signs. Further, by \eqref{eq : a powers} in Lemma \ref{triangles opposite signs}, the ratios of $a_k$-powers of the capped disks of the triangles in $\sA_3$ and the $a_k$-powers of related disks in $\sA_1$ and $\sA_2$, respectively, are equal, which means that there could exist a suitable capping disk on $L_3$ that make all triangles cancel in pairs. We next show there in fact is such a cap.

    Figure \ref{fig:c_kk Reeb chords} shows the location of the minimal Reeb chords $c_{kk}$. We pick the location on $L_1$ and then let the corresponing Reeb chords on $L_2$ and $L_3$ be obtianed by translation $\frac{2\pi}{9}$ and $\frac{4\pi}{9}$ units along the Reeb flow. Note also that this determines the evaluation maps at the negative punctures into the Reeb chord manifolds of the shorter mixed Reeb chords. 

We pick the location of the start point of the short mixed Reeb chord $a_{k,k+1}$ (on $\Lambda_k$) to lie close to the evaluation map, i.e., $c_{kk}^-$, and the location endpoint of the longer mixed Reeb chord $b_{k+1,k}$ (again on $\Lambda_k$) to lie close to the evaluation map, i.e., $c_{kk}^+$, see Figure \ref{fig:triangle-eval}. 

With these choices the images of evaluation maps at negative punctures of triangles with positive puncture at $c_{kk}$ lie in a small neighborhood of the Reeb chords at their negative ends except for the evaluation at the puncture $b_{13}$ in the triangle with positive puncture at $c_{11}$, the evaluation at $a_{12}$ in the triangle with positive puncture at $c_{22}$, and the evaluation at $a_{23}$ in the triangle with positive puncture at $c_{33}$. We pick perturbing Morse functions so that flow lines from evaluations in small neighborhoods of Reeb chord minima stay inside these small neighborhoods, see Figures \ref{fig:c_11 triangles}, \ref{fig:c_22 triangles}, and \ref{fig:c_33 triangles}. 

For the evaluation maps not in a neighborhood we pick flow lines as follows: $\beta$ for $b_{13}$ in Figure \ref{fig:c_11 triangles} and its translates, $\gamma$ for $a_{12}$ in Figure \ref{fig:c_22 triangles} and $\delta$ for $a_{23}$ in Figure \ref{fig:c_33 triangles}. (We will discuss other possible choices at the end of the proof.) 

With triangles as described in Figures \ref{fig:c_11 triangles}, \ref{fig:c_22 triangles}, and \ref{fig:c_33 triangles} it is straightforward to check that if we choose capping disks with boundaries $\alpha_k$ as shown in Figure \ref{fig:cappingpaths}, then the pairs of disks with the same negative asymptotics cancel. Relation $\sA_1$ follows.

In order to derive $\sA_2$ and $\sA_3$ we note that in the argument for $\sA_1$ just given, the Morse flow $\beta$ in the Bott manifold of $a_{12}$ was chosen, then other Morse flows $\gamma$ and $\delta$ were determined by Reeb translation and the boundaries of the capping disks were taken as nearby paths with the opposite orientation. To derive $\sA_2$ and $\sA_3$, we simply repeat the argument replacing $\beta=\beta_1$ in Figure \ref{fig:differentcaps} by the flow lines $\beta_2$ for $\sA_2$ and $\beta_3$ for $\sA_3$ in Figure \ref{fig:differentcaps}. The proposition follows.   
\end{proof}

\begin{remark}
In the proof of Proposition \ref{cancelling triangles} we pick the Morse flow line $\beta$ and the corresponding capping path $\alpha_j$. We use three different choices $\beta=\beta_j$, $j=1,2,3$, in order to get an complete system of operator equations. Note however that there are other possible choices, in particular taking the Morse flow and capping path as primitive curves in the tori we get one operator equation for each relatively prime pair of integers $(i,j)$.  
\end{remark}

\begin{figure}
    \centering
    \includegraphics[width=0.8\linewidth]{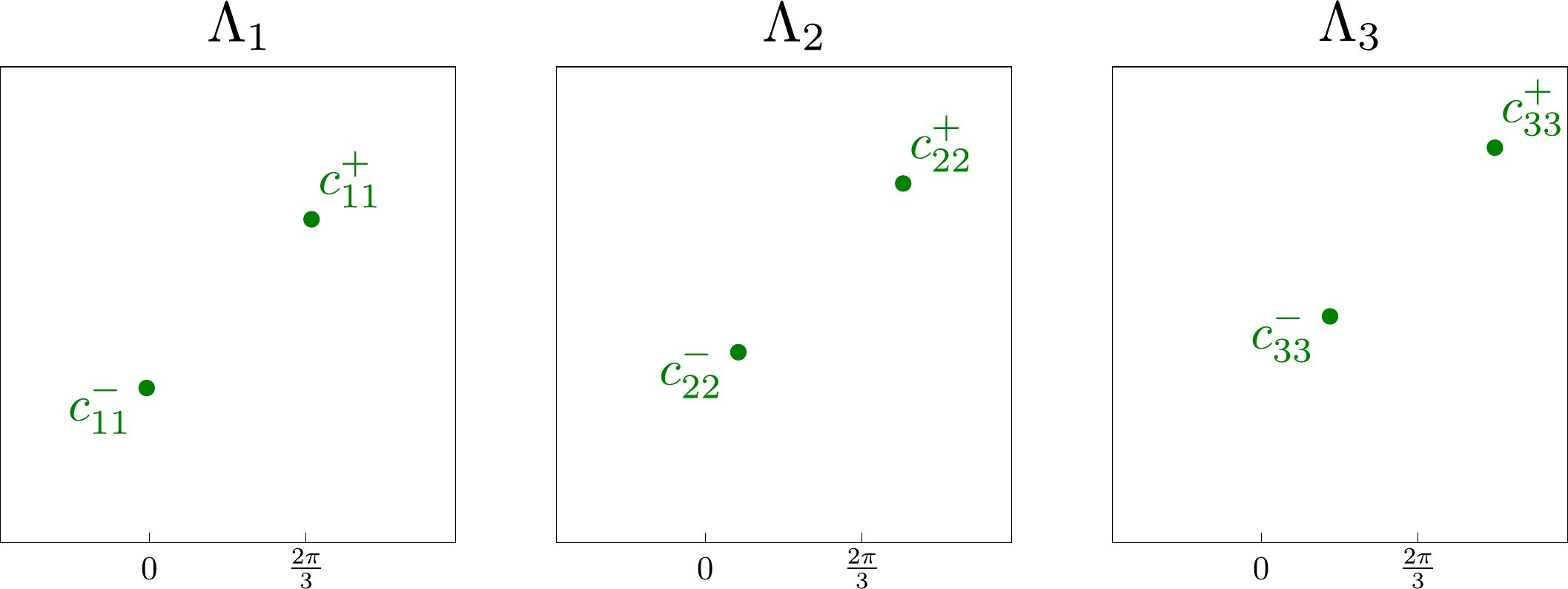}
    \caption{The location of the minimal action Reeb chords $c_{kk}$ which is also the image of the evaluation maps at the positive puncture of the triangles with positive puncture at $c_{kk}$ and negative punctures at $b_{k,k+2}$ and $a_{k+2,k+3}$ or at $a_{k,k+1}$ and $b_{k+1,k+3}$.}
    \label{fig:c_kk Reeb chords}
\end{figure}

\begin{figure}
    \centering
    \includegraphics[width=0.8\linewidth]{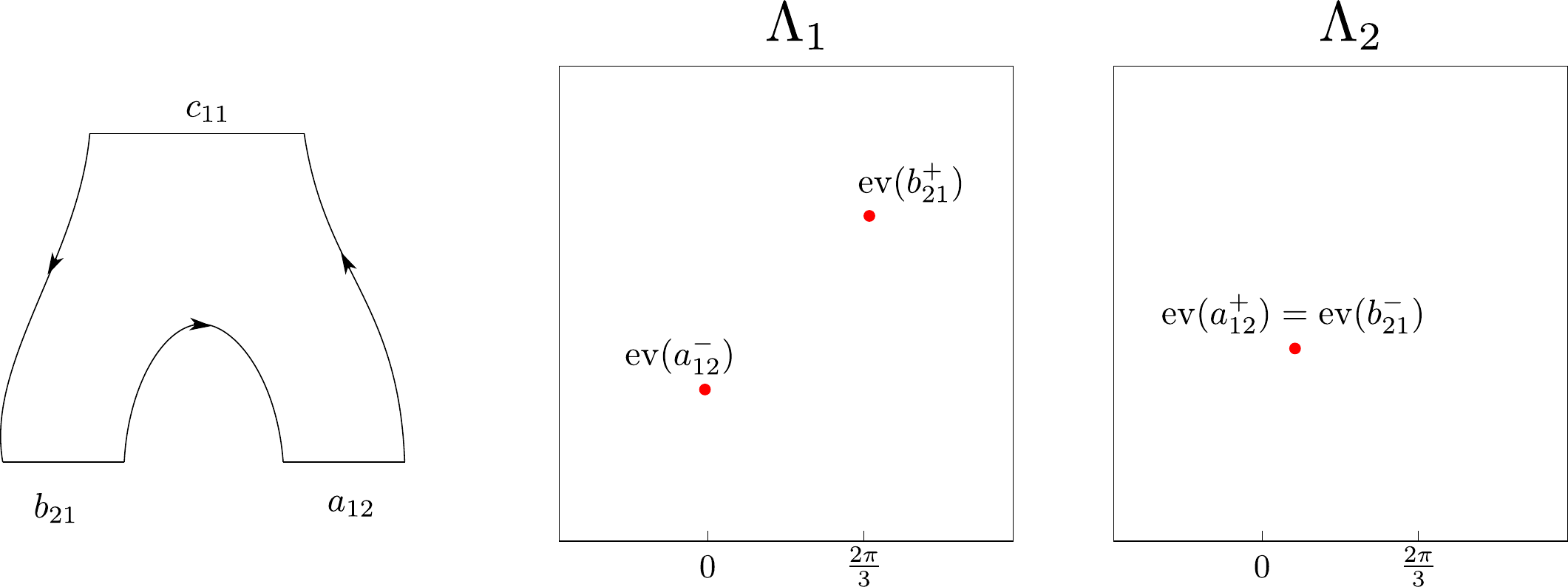}
    \caption{The evaluation map $\mathrm{ev}$ of the Reeb chords at the negative end in the triangle with positive puncture $c_{11}$ and negative punctures $a_{12}$ and $b_{21}$. Since the triangle lies in a fiber, the negative end point of $a_{12}$ agrees with the negative endpoint of $c_{11}$ and the positive endpoint of $b_{21}$ with the positive endpoint of $c_{11}$. The locations of the positive endpoint of $a_{12}$ and the negative endpoint of $b_{21}$ agree and are determined by following the respective Reeb chords. Evaluation maps of other triangles are determined in a similar way. }
    \label{fig:triangle-eval}
\end{figure}

\begin{figure}
    \centering
    \includegraphics[width=.8\linewidth]{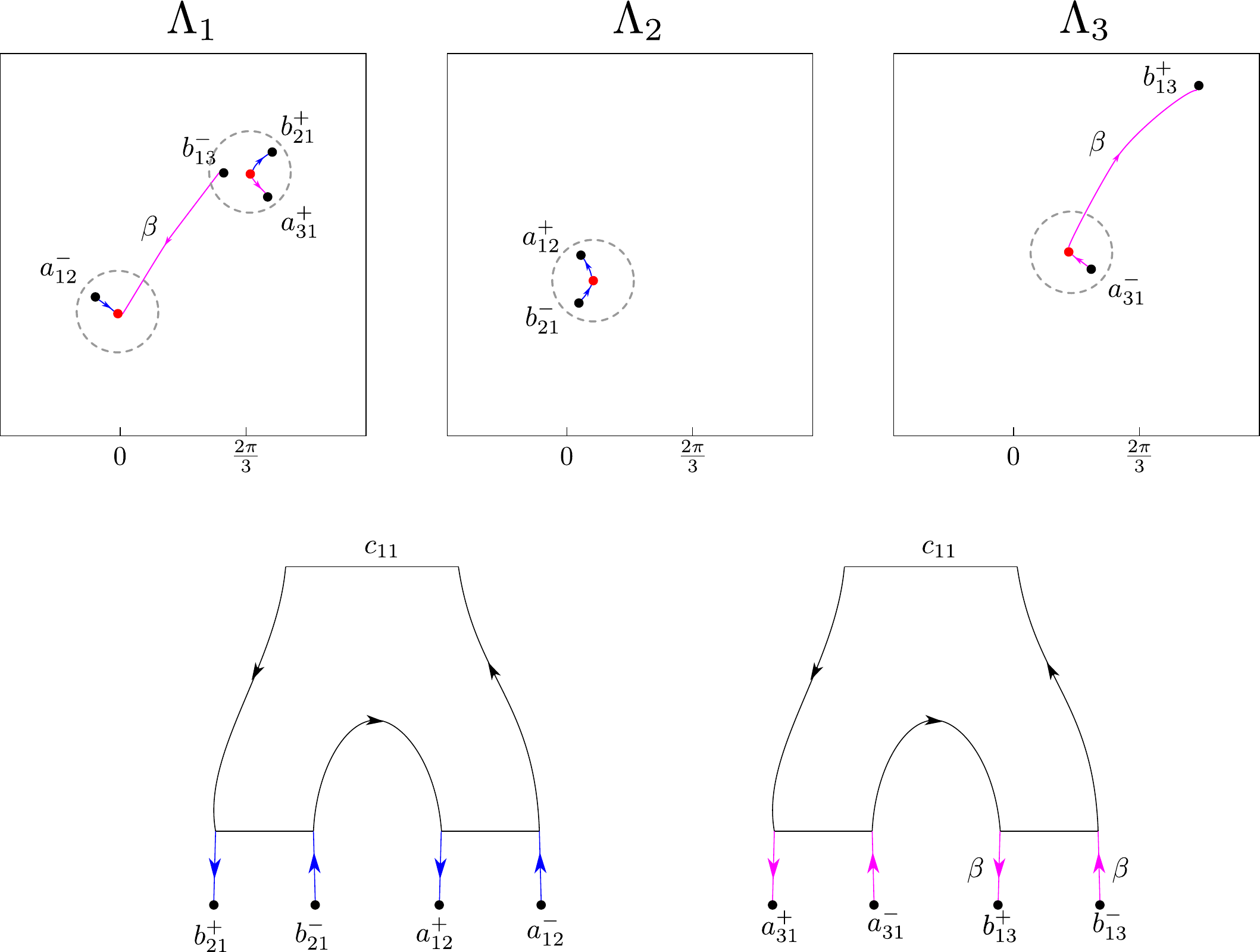}
    \caption{Triangles with positive puncture at $c_{11}$. Red dots indicate evaluation maps, see Figure \ref{fig:triangle-eval}, dashed circles indicate the neighborhoods of the evaluation maps, black dots the location of minimal chords $a_{12}$, $b_{13}$, $a_{31}$, and $b_{21}$, and curve segments indicate Morse flows. The actual punctured holomorphic disks are arbitrarily well approximated by the original disks with Morse flow segments attached. }
    \label{fig:c_11 triangles}
\end{figure}
\begin{figure}
    \centering
    \includegraphics[width=.8\linewidth]{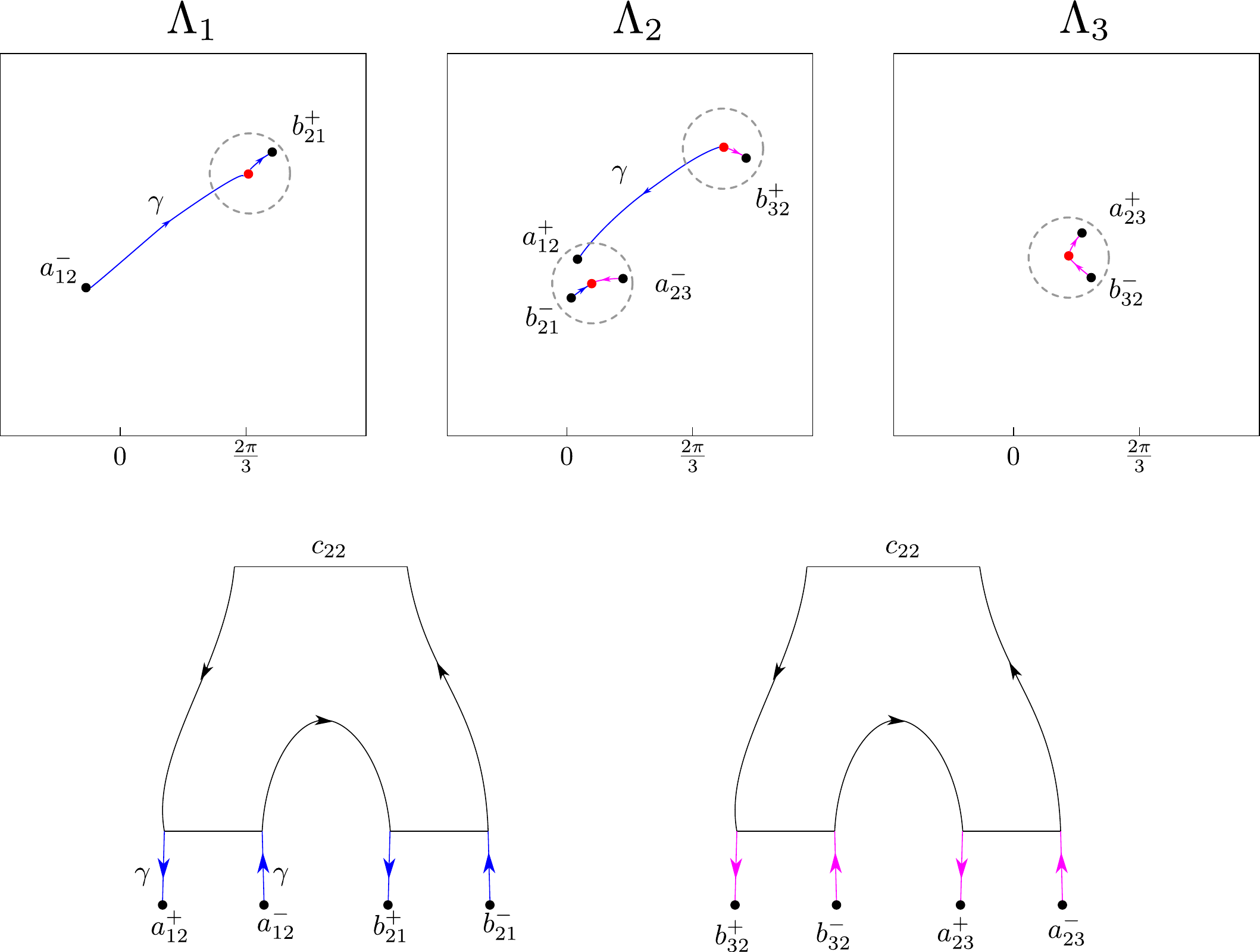}
    \caption{Triangles with positive puncture at $c_{22}$. Notation as in Figure \ref{fig:c_11 triangles}.}
    \label{fig:c_22 triangles}
\end{figure}

\begin{figure}
    \centering
    \includegraphics[width=.8\linewidth]{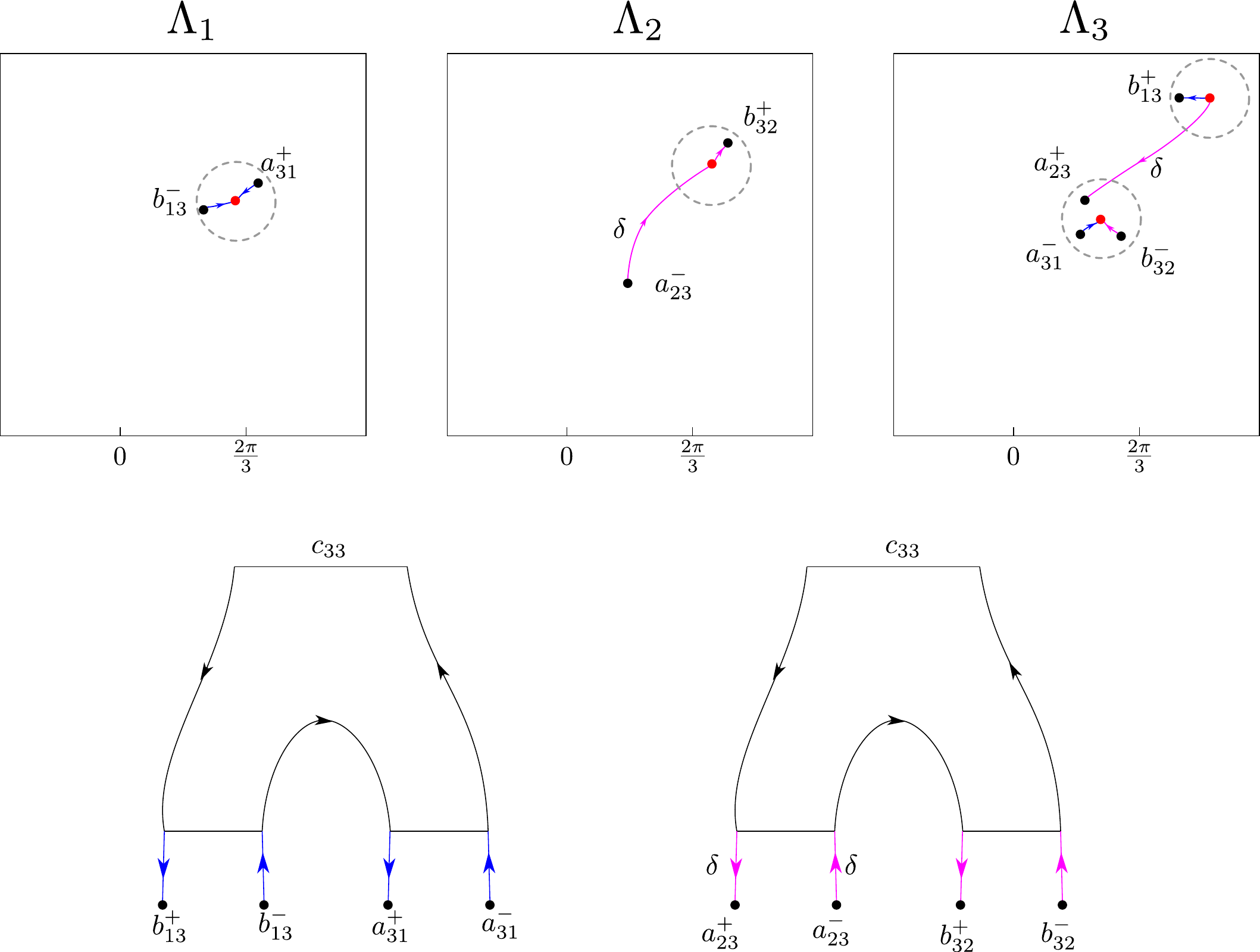}
    \caption{Triangles with postive puncture at $c_{33}$. Notation as in Figure \ref{fig:c_11 triangles}.}
    \label{fig:c_33 triangles}
\end{figure}

\begin{figure}
    \centering
    \includegraphics[width=.8\linewidth]{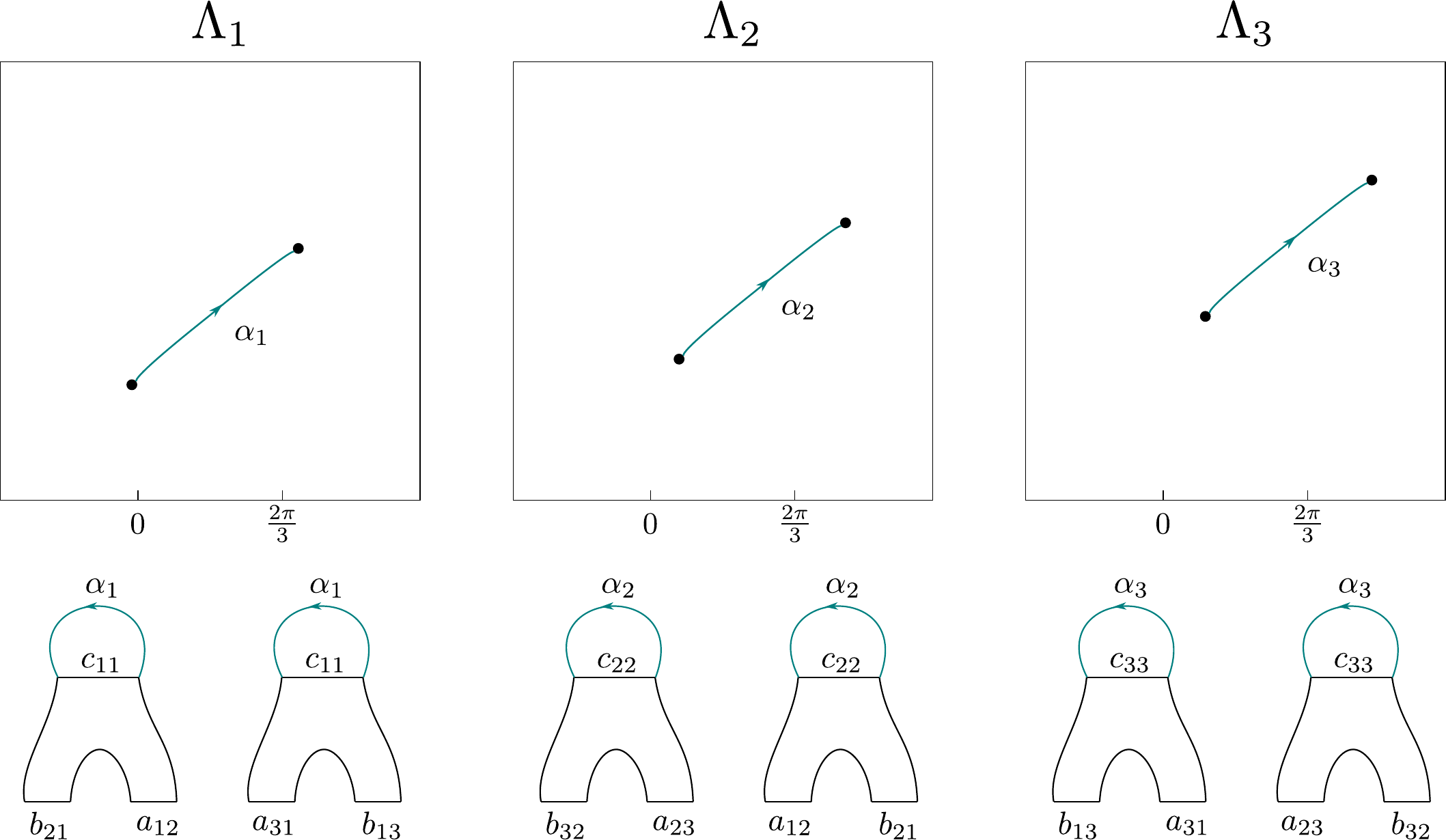}
    \caption{Capping disks for the positive punctures for cancellation of triangles.}
    \label{fig:cappingpaths}
\end{figure}

\begin{figure}
    \centering
    \includegraphics[width=0.35\linewidth]{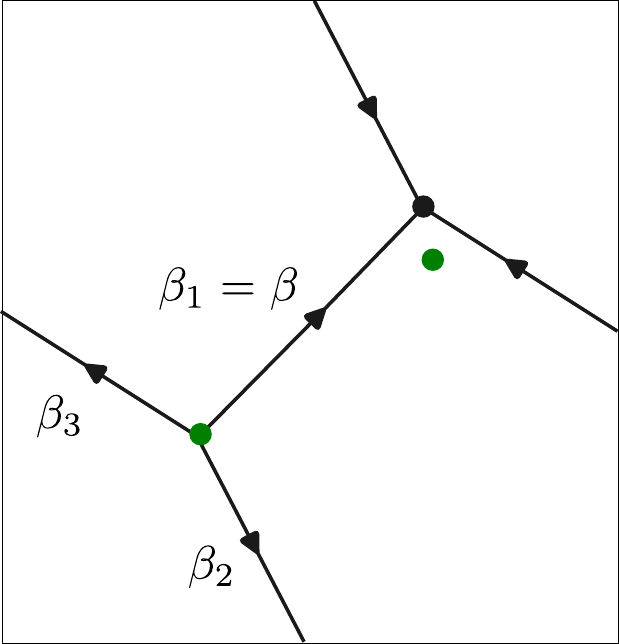}
    \caption{Different choices of Morse flow lines in Bott manifolds}
    \label{fig:differentcaps}
\end{figure}

\clearpage

\subsection{Proof of Theorem \ref{count at infinity intro}}

It follows from the discussion in Section \ref{strategy} that the operators $\sA_j$ as defined in Proposition \ref{cancelling triangles} satisfy $\sA_j \sZ_{\C^3, L_1, L_2, L_3} = 0$.  Thus it remains only to show that these operators are given by the explicit formulas of the introduction.  We will treat $\sA_1$; the remainder follow either similarly or by symmetry. 

We now impose the spin structures and framing choices of Section \ref{framing spin choices}; in particular, we choose longitudes of the $L_i$ as specified there.  
Then it follows from Proposition \ref{cancelling triangles} that 
   $$
    \sA_1 = \ast \sP_{0,0}^{(1)}  + \ast \sP_{1,0}^{(1)} + \ast \sP_{0,1}^{(1)}
    + \ast \sP_{-1,0}^{(2)}  + \ast \sP_{0,0}^{(2)} 
    + \ast \sP_{-1,1}^{(2)}
    + \ast \sP_{0,-1}^{(3)}  
    + \ast \sP_{1,-1}^{(3)} + \ast 
    \sP_{0,0}^{(3)}
    $$
where $\ast$ are signed monomials in the framing variables $a_1, a_2, a_3$. 

Let us observe that $\sZ_{\C^3, L_1, L_2, L_3}$, when expanded in the basis $\sW_{\lambda_1, \overline{\mu}_1} \otimes \sW_{\lambda_2, \overline{\mu}_2} \otimes \sW_{\lambda_3, \overline{\mu}_3}$, has only nonzero coefficients when all $\mu_i$ are empty.  Indeed, this is because any curve counted is a small perturbation of a map whose non-contracted components are multiple covers of disks in the coordinate axes.  In particular, the boundaries wind positively around the longitude of the $L_i$; this implies the vanishing of the $\mu_i$.  

Now writing $z_{\lambda_1, \lambda_2, \lambda_3}$ for the coefficient of 
$\sW_{\lambda_1, \emptyset} \otimes \sW_{\lambda_2, \emptyset} \otimes \sW_{\lambda_3, \emptyset}$,
we claim that we have the following leading terms: 
$$
    z_{\emptyset, \emptyset, \emptyset}  = 1, \qquad \qquad 
     z_{\emptyset, \emptyset, \ydiagram{1}}  =  - (q^{1/2} - q^{-1/2})^{-1}, 
    \qquad \qquad
    z_{ \emptyset,  \ydiagram{1}, \ydiagram{1}} =
    z_{\ydiagram{1}, \emptyset, \ydiagram{1}}  = 
      1 + (q^{1/2} - q^{-1/2})^{-2}.
$$
Indeed,  $z_{\emptyset, \emptyset, \emptyset}$ counts the moduli of maps from the empty curve which is a point with the positive orientation. By Lemma \ref{l : simple disks and annuli} and choice of $4$-chains, see Lemma \ref{l : vector filed for 4-chain}, $z_{\emptyset, \emptyset, \ydiagram{1}}$ is the count of a single embedded disk and we chose the spin structure to ensure the negative sign (by definition, in the \cite{SoB} formalism, a disk is counted by $(q^{1/2} - q^{-1/2})^{-1}$).  Finally, again by Lemma~\ref{l : simple disks and annuli}, $z_{ \emptyset,  \ydiagram{1}, \ydiagram{1}}$ has two contributions: the map from a disjoint union of two disks, which gives $(q^{1/2} - q^{-1/2})^{-2}$, and a map from the annulus  from gluing them, which gives $1$. 

Let us write $a_{i,j}^{(k)}$ for the coefficient of $\sP_{i,j}^{(k)}$, and rescale so that $a_{0, -1}^{(3)} = -1$.  We claim that then the sign in front of $a_{i,j}$ is always $(-1)^{i+j}$.  Indeed, geometrically, aside from capping path choices, global symmetries permute all nine disks.  It follows that any two disks have the same sign if and only if the restriction of the spin structure of the Lagrangians to their boundaries is the same.  The choice of spin structure we have made is such that the longitude and meridian of each torus carry the bounding spin structure.  This gives the sign $(-1)^{i+j}$. 

Now observe that unless $a_{1, -1}^{(3)} =  a_3$, 
$\sA_1 \cdot \sZ_{\C^3, L_1, L_2, L_3}$ would have a nonvanishing coefficient of $\sW_{\emptyset, \emptyset} \otimes \sW_{\emptyset, \emptyset} \otimes \sW_{\emptyset, \overline{ \ydiagram{1}}}$.
So, $a_{1, -1}^{(3)} =  a_3$. 

Then the coefficient of 
$\sW_{\emptyset, \emptyset} \otimes \sW_{\emptyset, \emptyset} \otimes \sW_{\emptyset, \emptyset}$
in $\sA_1\cdot \sZ_{\C^3, L_1, L_2, L_3}$ is
    $$0=(a_{0,0}^{(1)} + a_{1,0}^{(1)}) \bigcirc^{(1)}  + (a_{0,0}^{(2)} + a_{-1, 0}^{(2)})\bigcirc^{(2)} 
    + (a_{0,0}^{(3)} - a_3)\bigcirc^{(3)}.$$
Since $\sA_1\cdot \sZ_{\C^3, L_1, L_2, L_3}= 0$, we may subtract the above expression (for $0$) from $\sA_1$ without changing it.  
Thus:  
    $$
    \sA_1 =  a_{1,0}^{(1)}(\sP_{1,0}^{(1)} - \sP_{0,0}^{(1)}) + a_{0,1}^{(1)} \sP_{0,1}^{(1)}
    + a_{-1,0}^{(2)}(\sP_{-1,0}^{(2)} - \sP_{0,0}^{(2)}) + a_{-1,1}^{(2)}  \sP_{-1,1}^{(2)}
    - \sP_{0,-1}^{(3)} + a_3 \sP_{1,-1}^{(3)} + a_3 
    \sP_{0,0}^{(3)}
    $$
We compute that the coefficient
of $\sW_{\ydiagram{1}, \emptyset}
\otimes \sW_{\emptyset, \emptyset}
\otimes \sW_{\emptyset, \emptyset}$
in $\sA_1\cdot \sZ_{\C^3, L_1, L_2, L_3}$ is: 
\begin{equation} \label{solve for 1 terms}
0 = a_{0,1}^{(1)} - a_{1,0}^{(1)} a_1 + a_3^2 -1. 
\end{equation}
Finally, we use the fact that the $a_{i,j}^{(k)}$ are monomials.  Then from \eqref{solve for 1 terms}, we see
$\{a_{0,1}^{(1)}, - a_{1,0}^{(1)} a_1\} = \{1, -a_3^2\}$.  The condition on signs gives $a_{1,0}^{(1)} = - a_1^{-1}$, $a_{0,1}^{(1)} = - a_3^2$.  

Similarly from the coefficient of $\sW_{\emptyset, \emptyset}
\otimes \sW_{\ydiagram{1}, \emptyset}
\otimes  \sW_{\emptyset, \emptyset}$, namely: 
\begin{equation} \label{solve for 2 terms}
0 = a_2^{-1}(a_{-1,0}^{(2)} + a_{-1,1}^{(2)}) + a_3^2 - 1.
\end{equation}
We find $a_{-1,0}^{(2)} = - a_2 a_3^2$ and 
$a_{-1, 1}^{(2)} = a_2$.
This recovers the formula in the introduction and thus completes the proof of Theorem \ref{count at infinity intro}.\qed

\appendix
\addtocontents{toc}{\protect\setcounter{tocdepth}{1}}

\section{$U(1)$ specialization and augmentation variety}
In this section we record the specializations of formulas for the skein operator equations and partition functions to the $U(1)$ skein, and then their dequantizations.  In general, the $U(1)$ skein is a further quotient, splittng the first and second skein relations into two two-term relations, of the specialization of the HOMFLYPT skein at $a=q^{1/2}$.  

For $T^2 \times I$, the $U(1)$ skein is the quantum torus with relation $\hat y \hat x = q \hat x \hat y$.
The action of $\Sk(T^2 \times I)$ on $\Sk(\R^2 \times S^1)$ specializes to the polynomial representation of the quantum torus.  Explicitly, we have:  
\be
	\sP^{(i)}_{1,0} \mapsto \hat y_i,
	\quad
	\sP^{(i)}_{0,1} \mapsto \hat x_i,
	\quad
	a_i\mapsto q^{1/2},
	\quad
	\sW_{\lambda}^{(i)}\mapsto s_{\lambda}(x_i).
\ee
Note that $x_i$ is a single variable, so $s_\lambda(x_i) = 0$ if $\lambda$ has more than one row.
We will use the following notation for $q$-Pochhammers 
\be
	(x;q)_n := \prod_{s=0}^{n-1}(1-xq^s)\,.
\ee

The $U(1)$ specialization of the partition function of Corollary \ref{TV counts intro} is: 
\be\label{eq:vertex-abelian}
\begin{split}
	Z	
	& = \sum_{k_1, k_2, k_3\geq 0} \mathcal{T}_{(k_1),(k_2),(k_3)} x_1^{k_1} x_2^{k_2} x_3^{k_3} \\
	& = 
	\sum_{n_1, n_2, n_3}
	\frac{(-q^{1/2})^{n_1^2} x_1^{n_1}}{(q;q)_{n_1}}
	\frac{(-q^{1/2})^{n_2^2} x_2^{n_2}}{(q;q)_{n_2}}
	\frac{(-q^{1/2})^{n_3^2} x_3^{n_3}}{(q;q)_{n_3}}
	\\
	& 
	\times \Big[
	\sum_{m_{12}, m_{23}, m_{31}=0}^{1}
	(x_1 x_2)^{m_{12}}
	(x_2 x_3)^{m_{23}}
	(x_3 x_1)^{m_{31}}
	q^{n_1 m_{12} + n_2 m_{23} + n_3 m_{31} + m_{12}m_{23}m_{31}} 
	\\
	& \ \ +
	(q^{1/2}-q^{-1/2}) \, q^{n_1+n_2+n_3}\, x_1 x_2 x_3
	\Big]
\end{split}
\ee
The specialization of Theorem \ref{count at infinity intro} is $\hat A_i \cdot Z = 0$ for $i=1,2,3$, where
\be
\begin{split}
	&
	\hat A_1 = 
	(1-\hat y_1 - q^{3/2} \hat x_1) 
	- q^2 \hat y_2^{-1} (1-\hat y_2 - q^{-1/2} \hat x_2) 
	- q^{1/2} \hat x_3^{-1} (1 - \hat y_3 - q^{1/2}\hat x_3),
	\\
	&
	\hat A_2 = 
	(1-\hat y_2 - q^{3/2} \hat x_2) 
	- q^2 \hat y_3^{-1} (1-\hat y_3 - q^{-1/2} \hat x_3) 
	- q^{1/2} \hat x_1^{-1} (1 - \hat y_1 - q^{1/2}\hat x_1),
	\\
	&
	\hat A_3 = 
	(1-\hat y_3 - q^{3/2} \hat x_3) 
	- q^2 \hat y_1^{-1} (1-\hat y_1 - q^{-1/2} \hat x_1) 
	- q^{1/2} \hat x_2^{-1} (1 - \hat y_2 - q^{1/2}\hat x_2).
	\\
\end{split}
\ee

To dequantize, set $q=1$. This recovers the augmentation variety of the Legendrian DGA of $\Lambda_1 \sqcup \Lambda_2 \sqcup \Lambda_3$ (which can also be determined by a dequantized version of the arguments in Section \ref{sec : operator equation}):
\be
\begin{split}
	&
	A_1 = 
	(1- y_1 -  x_1) 
	-  y_2^{-1} (1- y_2 -  x_2) 
	-  x_3^{-1} (1 -  y_3 -  x_3)
	\\
	&
	A_2 = 
	(1- y_2 -  x_2) 
	-  y_3^{-1} (1- y_3 -  x_3) 
	-  x_1^{-1} (1 -  y_1 -  x_1)
	\\
	&
	A_3 = 
	(1- y_3 -  x_3) 
	-  y_1^{-1} (1- y_1 -  x_1) 
	-  x_2^{-1} (1 - y_2 -  x_2)
	\\
\end{split}
\ee

We recall  that the augmentation variety parametrizes objects in the Fukaya category of $\C^3$ stopped at $\Lambda_1 \sqcup \Lambda_2 \sqcup \Lambda_3$ with $\mathrm{Hom}$-spaces to the linking disks of the $\Lambda_i$ are all one dimensional in degree zero.  Indeed, since $\C^3$ is subcritical, by \cite{GPS2} the linking disks generate; by \cite{Ekholm-Lekili}, the endomorphism algebra of the linking disks is equivalent to the Legendrian DGA. It is in this sense that the $\sA_i$ determine a skein-valued quantization of the mirror.

\section{Other fillings}
The Legendrian $\Lambda_1 \sqcup \Lambda_2 \sqcup \Lambda_3$ admits, up to cyclic permutation,  four types of fillings by toric branes denoted Filling 1, 2, 3, and 4 in Figure \ref{fig:fillings}. In the main body of the article we studied the most interesting of these, namely Filling 1. 

For the other fillings, note that the argument used to derive the operator equation in Proposition~\ref{cancelling triangles}
does neither depend on the choice of filling, nor
on the choice of 4-chain.  (Or in the language of~\cite{Ekholm:2024ceb}, the argument could have been made in the `worldsheet skein' of the symplectization of the contact boundary.)  However, when we write the equation for other fillings, the expression will be different.  There are two reasons for this.  The first is that we choose the basis of each torus $H_1(\Lambda_i)$ to ensure 
that $\sP_{1,0}$ always corresponds to the contractible cycle and $\sP_{0,1}$ to the framed generator of homology with framing induced by the vector field described in the main text.  That is, it depends on the filling.  

The effect of this can be summarized as follows: ignoring all powers of $a$, when moving brane $L_\alpha$ counter-clockwise by one step among the three legs shown in Figure \ref{fig:fillings}, we transform the operator equation by: 
\be\label{eq:ST}
	\sP_{i,j}^{(\alpha)} \mapsto (-1)^i \, \sP^{(\alpha)}_{-i-j,i}. 
\ee
The element of $SL(2,\mathbb{Z})$ 
which takes $(i, j) \mapsto (-i-j, -i)$
is a cube root of the identity, and sometimes denoted $T^{-1} S$ in the literature. 
The $a$ powers in each case in principle depend on the 4-chain intersections, and in practice can be determined by a similar argument as we used for Filling 1.  

In all cases it is possible to solve the resulting recursion uniquely given an initial condition (below called `ansatz') which can in turn be easily established from the geometry of the setup. The arguments are simpler than for Filling 1; in this appendix, we content ourselves with recording without proof the resulting formulas. 

Let us introduce the following notation for the skein-valued partition functions of the disk, the (framed) anti-annulus and the twisted anti-annulus
\be
\begin{split}
	\sPsi^{(\alpha)}_{\mathrm{di}}
	& := 
	\sum_\lambda \left(\prod_{\ydiagram{1}\in\lambda} \frac{-q^{-c(\ydiagram{1})}}{q^{h(\ydiagram{1})}-q^{-h(\ydiagram{1})}}\right)  \sW^{(\alpha)}_{\lambda,\emptyset}
	= \sum_\lambda s_\lambda(q^{-\rho})  \sW^{(\alpha)}_{\lambda,\emptyset}\,,
	\\
	\sPsi^{(\alpha\beta)}_{{\mathrm{aa};(f_1, f_2)}} &:= \sum_{\lambda}  \ ((-1)^{|\lambda|}q^{\kappa(\lambda)})^{f_1-f_2} (-1)^{|\lambda|}
	\sW_{\lambda,\emptyset}^{(1)}\otimes \sW_{\lambda^t,\emptyset}^{(2)}
	\\
	\sPsi^{(\alpha\beta)}_{\widetilde{\mathrm{aa}}} & := \sum_{\lambda} (-1)^{|\lambda|} \sW^{(\alpha)}_{\lambda,\emptyset}\otimes \sW^{(\beta)}_{\emptyset,\lambda^t}\,,
\end{split}
\ee
where $q^{-\rho} = (q^{1/2}, q^{3/2}, q^{5/2},\dots)$ is the infinite set of variables giving the principal specialization of Schur functions.
We also introduce the following products
\be
\begin{split}
	\sW_{\lambda,\emptyset}\star_f \sW_{\mu,\emptyset} & : = \sum_{\nu} q^{\frac{f}{2}(\kappa(\nu) - \kappa(\lambda)-\kappa(\mu))} \sW_{\nu,\emptyset}\,, \qquad (f\in \mathbb{Z})
	\\
	\sW_{\lambda,\emptyset}\odot \sW_{\emptyset,\bar\mu} & := \sW_{\lambda,\bar\mu}\,.
\end{split}
\ee
Here, we sometimes simplify notation and wtite $\star:=\star_0$. Note that the product $\star_f$ is associative only when $f=0$.

We then observe that the two-brane specialization of $\sZ$ from Theorem \ref{count at infinity intro} admits the following description
\be\label{eq:Z-vertex-2brane}
\begin{split}
	\sZ^{(12)} & = 
	\sum_{\lambda_1,\lambda_2}
	\mathcal{T}_{\lambda_1\lambda_2\emptyset}(q) \sW^{(1)}_{\lambda_1,\emptyset}\otimes \sW^{(2)}_{\lambda_2,\emptyset}
	=
	\sPsi_{\rm{aa};(-1,0)}^{(12)}\star_{(-1,0)}(\sPsi^{(1)}_{\rm{di}}\otimes \sPsi^{(2)}_{\rm{di}})
	\\
	& = 
	\sum_{\alpha_1,\alpha_2,\beta}
	s_{\alpha_1}(q^{-\rho})\, s_{\alpha_2}(q^{-\rho})
	q^{-\kappa(\beta)/2}
	(\sW^{(1)}_{\alpha_1}\star_{-1}\sW^{(1)}_{\beta})\otimes (W^{(2)}_{\alpha_2}\star_{0}W^{(2)}_{\beta^t})
	\\
	& = 
	\sum_{\lambda_1,\lambda_2} q^{-\kappa(\lambda_2)}H_{\lambda_1,\lambda_2^t}\;\sW_{\lambda_1,\emptyset}^{(1)}\otimes\sW_{\lambda_2,\emptyset}^{(2)}\,,
\end{split}
\ee
where
 $H_{\alpha,\beta}=H_{\beta,\alpha}$ coincides with the extremal HOMFLYPT polynomial of the Hopf link, as computed in \cite{Ekholm:2024ceb}. It is given by
\be
\begin{split}
	H_{\lambda,\mu^t} 
	= 
	q^{(\kappa(\mu)-\kappa(\lambda))/2}
	\sum_{\eta}
	s_{\lambda^t/\eta}(q^{-\rho})
	s_{\mu/\eta}(q^{-\rho})
	= 
	q^{-\kappa(\lambda)/2}
	s_{\lambda^t}(q^{-\rho})
	s_{\mu^t}(q^{-\rho-\lambda})\,.
\end{split}
\ee

\subsection{Filling 1}
The partition function is the one given in the main text. We restate it here for completeness:
\be
	\sZ_{\mathrm{F}1} = \sum_{\lambda_1,\lambda_2,\lambda_3}
	\mathcal{T}_{\lambda_1,\lambda_2,\lambda_3}(q)
	\,
	\sW^{(1)}_{\lambda_1,\emptyset}
	\otimes
	\sW^{(2)}_{\lambda_2,\emptyset}
	\otimes
	\sW^{(3)}_{\lambda_3,\emptyset},
\ee
where
\be
\begin{split}
	\mathcal{T}_{\lambda_1,\lambda_2,\lambda_3}(q)
	& = 
	\frac{q^{-(\kappa(\lambda_2)+\kappa(\lambda_3))/2}}{H_{\lambda_2,\emptyset}}\,
	\sum_{\mu_1,\mu_3,\nu} c^{\lambda_1}_{\nu \mu_1}c^{\lambda_3^t}_{\nu \mu_3^t}
	H_{\lambda_2^t,\mu_1}\,
	H_{\lambda_2,\mu_3^t}
	\\
	& = q^{-(\kappa(\lambda_1)+\kappa(\lambda_3))/2}
	s_{\lambda_3^t}(q^{-\rho})
	\sum_{\eta}
	s_{\lambda_1^t/\eta}(q^{-\rho-\lambda_3})
	s_{\lambda_2/\eta}(q^{-\rho-\lambda_3^t}).
\end{split}
\ee
It is the unique solution to the skein-valued recursion defined by the following operators with initial conditions corresponding to non-negative homology classes in all three Lagrangians
\be\label{eq:Reeb12STST3-A-relations}
\begin{split}
	\sA_{1}^{(\mathrm{F}1)}
	&=
	\left(
	\sP_{0,0}^{(1)}  - \sP_{1,0}^{(1)} -  a_{1} a_3^2  \,\sP_{0,1}^{(1)}
	\right)
	- a_1 a_2 a_3^2 
	\left(
	\sP_{-1,0}^{(2)}  - \sP_{0,0}^{(2)} -  a_3^{-2}  \,\sP_{-1,1}^{(2)}
	\right)
	\\
	&
	- a_1
	\left(
	\sP_{0,-1}^{(3)}  - a_3 \sP_{1,-1}^{(3)} -  a_3  \,\sP_{0,0}^{(3)}
	\right)
	\\
	\sA_{2}^{(\mathrm{F}1)}
	&=
	\left(
	\sP_{0,0}^{(2)}  - \sP_{1,0}^{(2)} -  a_{2} a_1^2  \,\sP_{0,1}^{(2)}
	\right)
	- a_2 a_3 a_1^2 
	\left(
	\sP_{-1,0}^{(3)}  - \sP_{0,0}^{(3)} -  a_1^{-2}  \,\sP_{-1,1}^{(3)}
	\right)
	\\
	&
	- a_2
	\left(
	\sP_{0,-1}^{(1)}  - a_1 \sP_{1,-1}^{(1)} -  a_1  \,\sP_{0,0}^{(1)}
	\right)
	\\
	\sA_{3}^{(\mathrm{F}1)}
	&= 
	 \left(
	\sP_{0,0}^{(3)}  - \sP_{1,0}^{(3)} -  a_{3} a_2^2  \,\sP_{0,1}^{(3)}
	\right)
	- a_3 a_1 a_2^2 
	\left(
	\sP_{-1,0}^{(1)}  - \sP_{0,0}^{(1)} -  a_2^{-2}  \,\sP_{-1,1}^{(1)}
	\right)
	\\
	&
	- a_3
	\left(
	\sP_{0,-1}^{(2)}  - a_2 \sP_{1,-1}^{(2)} -  a_2  \,\sP_{0,0}^{(2)}
	\right)
\end{split}
\ee

\subsection{Filling 2}
For this filling, the skein-valued partition function is
\be\label{eq:ZReeb-12ST3-reduced}
\begin{split}
	\sZ_{\mathrm{F}2} 
	& = 
	\sPsi^{(12)}_{\widetilde{\mathrm{aa}}}\odot \sZ^{(23)}
	= 
	\sum_{\lambda_1,\lambda_2,\lambda_3}
	(-1)^{|\lambda_1|}
	q^{- \kappa(\lambda_3)/2}
	\,
	H_{\lambda_2, \lambda_3^t}
	\sW^{(1)}_{\lambda_1,\emptyset}
	\otimes
	\sW^{(2)}_{\lambda_2,\lambda_1^t}
	\otimes
	\sW^{(3)}_{\lambda_3,\emptyset}
	\\
\end{split}
\ee
where $\sZ^{(\alpha\beta)}$ is given by \eqref{eq:Z-vertex-2brane} with a cyclic shift of labels $1\to 2\to 3\to 1$.
This is the unique solution to the following operator equations (which are obtained from $\sA_i$ by application of \eqref{eq:ST} to $\Sk(\Lambda_3)$)
\be\label{eq:Reeb12ST3-A-relations}
\begin{split}
	&
	\sA_{1}^{(\mathrm{F}2)} =  \left(\sP_{0,0}^{(1)}-\sP_{1,0}^{(1)}-a_1 a_2^2  \sP_{0,1}^{(1)}\right)
	+a_1 a_2\left(\sP_{0,0}^{(2)}-\sP_{1,0}^{(2)}-a_2 \sP_{0,1}^{(2)}\right)
	\\
	&
	-a_1 a_2^2 a_3\left(\sP_{-1,0}^{(3)}-\sP_{0,0}^{(3)}- \sP_{-1,1}^{(3)}\right)\\
	&
	\sA_{2}^{(\mathrm{F}2)} = a_1 a_2^2 a_3^2 \left( \sP^{(1)}_{-1,0} - \sP^{(1)}_{0,0} - a_3^{-2} \sP^{(1)}_{-1,1} \right)
	+ a_2 a_3^2 \left( \sP^{(2)}_{-1,0} - \sP^{(2)}_{0,0} - a_3^{-2} \sP^{(2)}_{-1,1} \right)
	\\
	&
	+ \left( \sP^{(3)}_{0,-1} - a_3 \sP^{(3)}_{1,-1} -  a_3 \sP^{(3)}_{0,0} \right)
	\\
	&
	\sA_{3}^{(\mathrm{F}2)} = a_3 \left( \sP^{(1)}_{0,-1} - a_1 \sP^{(1)}_{1,-1} - a_1 a_2^2  \sP^{(1)}_{0,0} \right)
	+ a_3 \left( \sP^{(2)}_{0,-1} - a_1^2 a_2  \sP^{(2)}_{1,-1} - a_2  \sP^{(2)}_{0,0} \right)
	\\
	&
	- \left( \sP^{(3)}_{0,0} - \sP^{(3)}_{1,0} - a_1^2 a_2^2 a_3 \sP^{(3)}_{0,1} \right)
	\\
\end{split}
\ee
and that satisfies the initial conditions of the following ansatz
\be
\begin{cases}	
    \sZ_{\mathrm{F}2} = \quad &\sum_{\lambda_1,\lambda_2,\bar\mu_2,\lambda_3} 
	\mathcal{U}^{(\mathrm{F}2)}_{\lambda_1,\lambda_2,\bar\mu_2,\lambda_3}(q)
	\sW^{(1)}_{\lambda_1,\emptyset}\otimes \sW^{(2)}_{\lambda_2,\bar\mu_2}\otimes \sW^{(3)}_{\lambda_3,\emptyset},\\

     & \vspace{-2mm} \\
    
	&  \mathcal{U}^{(\mathrm{F}2)}_{\lambda_1,\lambda_2,\bar\mu_2,\lambda_3}(q)=0 \text{ if } |\bar\mu_2| > |\lambda_1|\,.
\end{cases}
\ee

The $U(1)$ specialization of the partition function is
\be\label{eq:Reeb21-ST-Z-U1}
\begin{split}
	Z_{\mathrm{F}2}
	& = \sum_{n_1, n_2, n_3\geq 0}\sum_{m_1, m_2,m_3=0}^{1}
	q^{- m_1 (n_2-m_3)- m_2 (n_1+m_3)}
	\\
	& \qquad\times \left(-\frac{x_1}{x_2}\right)^{m_3} 
	(x_2 x_3)^{m_1}
	(q x_3 x_1)^{m_2}
	\frac{(q^{1/2} x_1)^{n_1}}{(q;q)_{n_1}}
	\frac{(q^{1/2} x_2)^{n_2}}{(q;q)_{n_2}}
	\frac{(q^{1/2} x_3)^{n_3}}{(q;q)_{n_3}}
\end{split}
\ee
while $\sA_i^{(\mathrm{F}2)}$ specializes to the following 
\be
\begin{split}
	& 
	\hat A_{1}^{(\mathrm{F}2)} = (1-\hat y_1 - q^{3/2} \hat x_1) 
	+ q (1-\hat y_2 - q^{1/2}\hat x_2)
	-q^2 \hat y_3^{-1} (1-\hat y_3-q^{1/2} \hat x_3),
	\\
	&
	\hat A_{2}^{(\mathrm{F}2)} =q^{5/2} \hat y_1^{-1} (1-\hat y_1 - q^{-1/2} \hat x_1) 
	+ q^{3/2} \hat y_2^{-1}  (1-\hat y_2 - q^{-1/2} \hat x_2)
	+ \hat x_3^{-1} (1-\hat y_3-q^{1/2} \hat x_3),
	\\
	&
	\hat A_{3}^{(\mathrm{F}2)} =q^{1/2}\hat x_1^{-1} (1-\hat y_1 - q^{3/2} \hat x_1) 
	+  q^{1/2}\hat x_2^{-1}  (1-q \hat y_2 - q^{1/2} \hat x_2)
	-  (1-\hat y_3-q^{5/2} \hat x_3).
	\\
\end{split}
\ee

\subsection{Filling 3}
For this filling, the skein-valued partition function is
\be\label{eq:ZReeb-12STST3-reduced}
\begin{split}
	\sZ_{\mathrm{F}3} 
	& = 
	\sPsi^{(12)}_{\widetilde{\mathrm{aa}}}\odot \sZ^{(32)}
	= 
	\sum_{\lambda_1,\lambda_2,\lambda_3}
	(-1)^{|\lambda_1|}
	q^{-\kappa(\lambda_2)/2}
	\,
	H_{\lambda_3,\lambda_2^t}
	\sW^{(1)}_{\lambda_1,\emptyset}
	\otimes
	\sW^{(2)}_{\lambda_2,\lambda_1^t}
	\otimes
	\sW^{(3)}_{\lambda_3,\emptyset},
	\\
\end{split}
\ee
where $\sZ^{(32)}$ is given by \eqref{eq:Z-vertex-2brane} with the change of labels $2\to2,  1\leftrightarrow 3$.
This is the unique solution to the following operator equations (which are obtained from $\sA_i$ by applying \eqref{eq:ST} twice to $\Sk(\Lambda_3)$)
\be\label{eq:Reeb12STST3-A-relations}
\begin{split}
	&
	\sA_{1}^{(F3)} = \left(\sP_{0,0}^{(1)}-\sP_{1,0}^{(1)}-a_1 a_2^2 a_3^2 \sP_{0,1}^{(1)}\right)
	+a_1 a_2\left(\sP_{0,0}^{(2)}-\sP_{1,0}^{(2)}-a_2 a_3^2 \sP_{0,1}^{(2)}\right)
	\\
	&
	-a_1 a_2^2 \left(\sP_{0,-1}^{(3)}-a_3 \sP_{1,-1}^{(3)}- a_3 \sP_{0,0}^{(3)}\right)
	\\
	& 
	\sA_{2}^{(F3)} = a_1 a_2^2 a_3^2 \left( \sP^{(1)}_{-1,0} - \sP^{(1)}_{0,0} -  \sP^{(1)}_{-1,1} \right)
	+ a_2 a_3^2 \left( \sP^{(2)}_{-1,0} - \sP^{(2)}_{0,0} -  \sP^{(2)}_{-1,1} \right)
	\\
	&
	- a_3 \left( \sP^{(3)}_{0,0} -  \sP^{(3)}_{1,0} -  a_3 \sP^{(3)}_{0,1} \right)
	\\
	& 
	\sA_{3}^{(F3)} = \left( \sP^{(1)}_{0,-1} - a_1 \sP^{(1)}_{1,-1} - a_1 a_2^2  \sP^{(1)}_{0,0} \right)
	+ \left( \sP^{(2)}_{0,-1} - a_1^2 a_2  \sP^{(2)}_{1,-1} - a_2  \sP^{(2)}_{0,0} \right)
	\\
	&
	+ a_1^2 a_2^2 a_3 \left( \sP^{(3)}_{-1,0} - \sP^{(3)}_{0,0} - a_1^{-2} a_2^{-2} \sP^{(3)}_{-1,1} \right)
	\\
\end{split}
\ee
and that satisfies the initial conditions of the following ansatz
\be
\begin{cases}
    \sZ_{\mathrm{F}3} = \quad &\sum_{\lambda_1,\lambda_2,\bar\mu_2,\lambda_3} 
	\mathcal{U}^{(\mathrm{F}3)}_{\lambda_1,\lambda_2,\bar\mu_2,\lambda_3}(q)
	\sW^{(1)}_{\lambda_1,\emptyset}\otimes \sW^{(2)}_{\lambda_2,\bar\mu_2}\otimes \sW^{(3)}_{\lambda_3,\emptyset},\\
	
     & \vspace{-2mm} \\
     
	& \mathcal{U}^{(\mathrm{F}3)}_{\lambda_1,\lambda_2,\bar\mu_2,\lambda_3}(q)=0 \text{ if } |\bar\mu_2| > |\lambda_1|\,.
\end{cases}
\ee

The $U(1)$ specialization of the partition function is
\be\label{eq:Reeb21-ST2-Z-U1}
\begin{split}
	Z_{\mathrm{F}3}
	& = \sum_{n_1, n_2, n_3\geq 0}\sum_{m_1, m_2,m_3=0}^{1}
	q^{-(m_1+m_2) n_3}
	q^{
	-2 m_1 m_2 -  m_2 m_3 +  m_1 m_3
	}
	\\
	& \qquad\times \left(-\frac{x_1}{x_2}\right)^{m_3} 
	(x_2 x_3)^{m_1}
	(q x_3 x_1)^{m_2}
	\frac{(q^{1/2} x_1)^{n_1}}{(q;q)_{n_1}}
	\frac{(q^{1/2} x_2)^{n_2}}{(q;q)_{n_2}}
	\frac{(q^{1/2} x_3)^{n_3}}{(q;q)_{n_3}}
\end{split}
\ee
while $\sA_i^{(\mathrm{F}3)}$ specialize to the following 
\be
\begin{split}
	& 
	\hat A_{1}^{(\mathrm{F}3)} = 
	(1-\hat y_1 - q^{5/2} \hat x_1) 
	+ q (1-\hat y_2 - q^{3/2}\hat x_2)
	-q^{3/2} \hat x_3^{-1} (1-\hat y_3-q^{1/2} \hat x_3)
	\\
	& 
	\hat A_{2}^{(\mathrm{F}3)} = 
	q^{5/2} \hat y_1^{-1} (1-\hat y_1 - q^{1/2} \hat x_1) 
	+ q^{3/2} \hat y_2^{-1}  (1-\hat y_2 - q^{1/2} \hat x_2)
	- q^{1/2} (1-\hat y_3-q^{1/2} \hat x_3)
	\\
	&
	\hat A_{3}^{(\mathrm{F}3)} = 
	\hat x_1^{-1} (1-\hat y_1 - q^{3/2} \hat x_1) 
	+  \hat x_2^{-1}  (1-q \hat y_2 - q^{1/2} \hat x_2)
	+ q^{5/2} \hat y_3^{-1} (1-\hat y_3-q^{-3/2} \hat x_3)
	\\
\end{split}
\ee

\subsection{Filling 4}
For this filling, the skein-valued partition function can be written as the product of unlinked twisted anti-annuli and disks:
\be\label{eq:Z123-Reeb-same-leg}
\begin{split}
	\sZ_{\mathrm{F}4} 
	&=  
	\sPsi_{\widetilde{\mathrm{aa}}}^{(12)}
	\star \sPsi_{\widetilde{\mathrm{aa}}}^{(23)}
	\star \sPsi_{\widetilde{\mathrm{aa}}}^{(13)}
	\star \left[\sPsi_{\mathrm{di}}^{(1)} \otimes \sPsi_{\mathrm{di}}^{(2)} \otimes \sPsi_{\mathrm{di}}^{(3)} \right]
	\\
	& = 
	\prod_{ij\in\{12,23,13\}}\left(\sum_{\mu_{ij}} (-1)^{|\mu_{ij}|} \sW^{(i)}_{\mu_{ij},\emptyset}\otimes \sW^{(j)}_{\emptyset,\bar \mu_{ij}^t} \right)
	\star
	\prod_{k\in\{1,2,3\}}
	\left(\sum_{\lambda_k}
	s_{\lambda_k}(q^{-\rho})
	\sW_{\lambda_k,\emptyset}^{(k)}
	 \right),
	\\
\end{split}
\ee
or, alternatively, in the more compact form using the `gluing' $\odot$ product:
\be
\begin{split}
	\sZ_{\mathrm{F}4}
	& = \sPsi_{\widetilde{\mathrm{aa}}}^{(12)}\odot \sPsi_{\widetilde{\mathrm{aa}}}^{(23)} \odot \sPsi^{(3)}_{\mathrm{di}} 
	\\
	& = 
	\sum_{\lambda_1,\lambda_2,\lambda_3}
	(-1)^{|\lambda_1|+|\lambda_2|} 
	s_{\lambda_3}(q^{-\rho})
	\,
	\sW^{(1)}_{\lambda_1,\emptyset}
	\otimes
	\sW^{(2)}_{\lambda_2,\lambda_1^t}
	\otimes
	\sW^{(3)}_{\lambda_3,\lambda_2^t}.
\end{split}
\ee

This is the unique solution to the following operator equations (which are obtained by applying the change of basis  \eqref{eq:ST}, twice in $\Sk(\Lambda_2)$ and once in $\Sk(\Lambda_3)$) 
\be\label{eq:Reeb-3-skein-recursions}
\begin{split}
	& 
	\sA_{1}^{(\mathrm{F}4)} = \left( \sP^{(1)}_{0,0} - \sP^{(1)}_{1,0} - a_1 a_2^2 a_3^2 \sP^{(1)}_{0,1} \right)
	+ a_1 a_2 \left( \sP^{(2)}_{0,0} - \sP^{(2)}_{1,0} - a_2 a_3^2 \sP^{(2)}_{0,1} \right)
	\\
	&+ a_1 a_2^2 a_3 \left( \sP^{(3)}_{0,0} - \sP^{(3)}_{1,0} - a_3 \sP^{(3)}_{0,1} \right)
	\\
	& 
	\sA_{2}^{(\mathrm{F}4)} = a_1 a_2^2 a_3 \left( \sP^{(1)}_{-1,0} - \sP^{(1)}_{0,0} -  \sP^{(1)}_{-1,1} \right)
	+ a_2 a_3 \left( \sP^{(2)}_{-1,0} - \sP^{(2)}_{0,0} -  \sP^{(2)}_{-1,1} \right)
	\\
	&+ \left( \sP^{(3)}_{-1,0} - \sP^{(3)}_{0,0} -  \sP^{(3)}_{-1,1} \right)
	\\
	& 
	\sA_{3}^{(\mathrm{F}4)} = \left( \sP^{(1)}_{0,-1} - a_1 \sP^{(1)}_{1,-1} - a_1 a_2^2 a_3^2 \sP^{(1)}_{0,0} \right)
	+ \left( \sP^{(2)}_{0,-1} - a_1^2 a_2  \sP^{(2)}_{1,-1} - a_2 a_3^2 \sP^{(2)}_{0,0} \right)
	\\
	&+ \left( \sP^{(3)}_{0,-1} - a_1^2 a_2^2 a_3  \sP^{(3)}_{1,-1} - a_3 \sP^{(3)}_{0,0} \right)
	\\
\end{split}
\ee
and that satisfies the initial conditions of the following ansatz
\be
\begin{cases}
\sZ_{\mathrm{F}4} \ = & 
	\sum_{\lambda_1,\lambda_2,\bar\mu_2,\lambda_3,\bar\mu_3} \;\;
	\mathcal{U}^{(\mathrm{F}4)}_{\lambda_1,\lambda_2,\bar\mu_2,\lambda_3,\bar\mu_3}(q)
	\sW^{(1)}_{\lambda_1,\emptyset}\otimes \sW^{(2)}_{\lambda_2,\bar\mu_2}\otimes \sW^{(3)}_{\lambda_3,\bar\mu_3},\\

     & \vspace{-2mm} \\

	& \mathcal{U}^{(F4)}_{\lambda_1,\lambda_2,\bar\mu_2,\lambda_3,\bar\mu_3}(q)=0 
	\text{ if } |\bar\mu_2| > |\lambda_1|
	\text{ or } |\bar\mu_3| > |\lambda_1|+|\lambda_2|.
\end{cases}
\ee

\begin{remark}
    In fact, relation $\sA_1^{(\mathrm{F}4)}$ alone determines the partition function with the given initial conditions,  
    and the three relations are related by a simultaneous change of capping paths at infinity. 
\end{remark}

\begin{remark}
    It is possible to derive the expression 
    $$
    \sZ_{\mathrm{F}4} \ = \ \sPsi_{\widetilde{\mathrm{aa}}}^{(12)}
	\star \sPsi_{\widetilde{\mathrm{aa}}}^{(23)}
	\star \sPsi_{\widetilde{\mathrm{aa}}}^{(13)}
	\star \left[\sPsi_{\mathrm{di}}^{(1)} \otimes \sPsi_{\mathrm{di}}^{(2)} \otimes \sPsi_{\mathrm{di}}^{(3)} \right]
    $$ 
    by a direct geometric argument, as follows. After some explicit perturbation of $L_k$, one can show that the only embedded holomorphic curves are: one disk ending on each Lagrangian, and one annulus stretching between each pair.  Moreover, all the boundaries of these curves are on parallel copies of the longitude, so commute in the skein. Finally, the full multiple cover contributions of disks and annuli were determined in \cite{Ekholm:2020csl} and \cite{ekholm2021coloredhomflypt}, respectively.  
\end{remark}

The $U(1)$ specialization of $\sZ_{\mathrm{F}4}$ is
\be\label{eq:3-branes-Z-Reeb-same-leg}
	Z_{\mathrm{F}4}= (1-x_1/x_2) (1-x_1/x_3) (1-x_2/x_3) (x_1;q)_\infty^{-1}(x_2;q)^{-1}_\infty(x_3;q)^{-1}_\infty
\ee
while the $U(1)$ specialization of $\sA_{i}^{(\mathrm{F}4)}$ is
\be\label{eq:3-branes-samel-eg-Z3-rec-U1}
\begin{split}
	&
	\hat A_1^{(\mathrm{F}4)} = 
	(1-\hat y_1 - q^{5/2} \hat x_1)
	+ q (1 - \hat y_2 - q^{3/2} \hat x_2)
	+ q^2 (1-\hat y_3 - q^{1/2} \hat x_3),
	\\
	& 
	\hat A_2^{(\mathrm{F}4)} = 
	\hat y_1^{-1}(1-\hat y_1 - q^{1/2} \hat x_1)
	+ q^{-1}  \hat y_2^{-1}(1 - \hat y_2 - q^{1/2} \hat x_2)
	+ q^{-2} \hat y_3^{-1} (1-\hat y_3 - q^{1/2} \hat x_3),
	\\
	& 
	\hat A_3^{(\mathrm{F}4)} = 
	\hat x_1^{-1}(1-\hat y_1 - q^{5/2} \hat x_1)
	+  \hat x_2^{-1}(1 - q \hat y_2 - q^{3/2} \hat x_2)
	+  \hat x_3^{-1}(1- q^2\hat y_3 - q^{1/2} \hat x_3).
\end{split}
\ee

\bibliographystyle{hplain}
\bibliography{biblio}

\end{document}